\documentclass[11pt, a4paper]{article}
\usepackage[left=2.5cm,right=3cm,top=2.5cm,bottom=2.5cm,includeheadfoot]{geometry}

\usepackage[ngerman, english]{babel}

\usepackage[utf8]{inputenc}

\usepackage{amsmath}
\usepackage{amsfonts}
\usepackage{amssymb}
\usepackage{amsthm} 
\usepackage{stmaryrd}
\usepackage{dsfont}
\usepackage{mathtools}

\usepackage[bf,hang,nooneline,font=footnotesize]{caption}							% formatting of figure and table descriptions
\usepackage{graphicx}																% insert graphic files
\usepackage{tabularx}																% improved table environment
\usepackage{multirow}																% multiple line cells in tables
\usepackage{booktabs}
\usepackage[labelformat=brace, font=scriptsize, textfont=it]{subcaption}

\theoremstyle{plain}
\newtheorem{theorem}{Theorem}[section]
\newtheorem{lemma}[theorem]{Lemma}

\newtheorem{proposition}[theorem]{Proposition}
\newtheorem{assumption}[theorem]{Assumption}

\theoremstyle{definition}
\newtheorem{definition}[theorem]{Definition}

\theoremstyle{remark}
\newtheorem{remark}[theorem]{Remark}

\usepackage{url} 

\usepackage{hyperref}
\usepackage{graphicx}
\usepackage{xcolor}
\usepackage{pgf}
\usepackage{tikz}
\usetikzlibrary{arrows,automata,positioning,calc}
\usepackage{array}
\usepackage[titletoc,title]{appendix}		

\usepackage{changes}

\newcommand{\defined}{\triangleq}
\newcommand{\de}{\mathop{}\!\mathrm{d}}

\DeclareMathOperator{\argmax}{argmax}

\newcommand\xqed[1]{%
	\leavevmode\unskip\penalty9999 \hbox{}\nobreak\hfill
	\quad\hbox{#1}}
\newcommand\remEnde{\xqed{$\triangle$}}

\begin{document}
	
	\title{Finite State Mean Field Games with Common Shocks}
	\author{
		Berenice Anne Neumann%
		\thanks{Trier University,
			Department IV -- Mathematics,
			Germany,
			e-mail: neumannb@uni-trier.de.}
		\and
		Frank T.\ Seifried%
		\thanks{Trier University,
			Department IV -- Mathematics, Germany,
			e-mail: seifried@uni-trier.de.}
	}
	\maketitle
	
	\begin{abstract}
		We present a novel framework for mean field games with finite state space and common noise, where the common noise is given through shocks that occur at random times. 
		We first analyze the game for up to $n$ shocks, in which case we are able to characterize mean field equilibria through a system of parameterized and coupled forward-backward equations.
		We establish existence and uniqueness of solutions to this system for small time horizons.
		In addition, we show that mean field equilibria for the $n$-shock setting constitute approximate equilibria for the corresponding mean field game with infinitely many common shocks.
		Our results are illustrated in a corruption detection model with random audits.\\[0.2cm]
	\end{abstract}
	
	\allowdisplaybreaks
	
	\section{Introduction}
	
	Mean field games were first introduced by Lasry and Lions \cite{LasryJapanese2007} and Huang et al.\ \cite{HuangNCE2006} in order to make games with a large number of players mathematically tractable. 
	Namely, it is well known that $N$-player games with a large number $N$ of players are usually intractable even if the game is symmetric and the interactions of players are weak in the sense that the other players influence a representative player only through their distribution.
	However, if one considers mean field games with a continuum of agents, which are suitable limits of such $N$-player games as $N\to\infty$, equilibria can be described by forward-backward systems of partial differential equations or stochastic differential equations.
	These insights led to significant advances in many applications in economics and finance, but also in epidemiology, computer science and sociology.
	
	A central assumption for the mean field approach is that the random factors are independent across players, leading to a deterministic evolution of the aggregate distribution of all players.
	However, this assumption is not satisfied in a number of applications, since in many situations players are exposed to common stochastic factors like market risks, weather conditions or political decisions. 
	Models with such common factors have also been analysed and are called mean field games \emph{with common noise}. 
	The analysis of these models is much more complicated, as in this case equilibria are stochastic because they depend on the realizations of the common noise.
	Hence, the characterization of equilibria becomes much more complex, since the Hamilton-Jacobi-Bellman equations that characterize optimal controls have to account for the common noise and thus become stochastic.  
	Against this background, it is not surprising that very few mean field games with common noise have been solved analytically or even numerically. 
	
	In this paper we introduce a new mean field game setup that exhibits a form of common noise present in many applications, but at the same time remains tractable. The baseline is a classical finite state model closely related to \cite{GomesConti2013} and \cite{CecchinProbabilistic2018}. In this model we introduce common noise in the form of common shock events. These events occur according to doubly stochastic Poisson processes whose intensity may in general depend on the population's distribution of states. Moreover, whenever a common shock occurs, this can change the current state of the players, the dynamics of the players' state dynamics as well as the players' rewards. This captures many classical forms of common noise in applications like a change of policies, natural disasters, innovations, etc.
	
	The main contributions of the paper are as follows:
	First, we consider a variant of our problem, where at most a fixed finite number of shocks can occur. In this case we prove that we can reduce the problem of finding an equilibrium to solving a \emph{coupled} system of parameterized forward-backward equations.
	Second, we show that on small time horizons a unique solution to this system exists and moreover, that it can be determined through a fixed point iteration.
	Third, we show that equilibria for the game with up to $n$ shocks constitute $\epsilon_n$-equilibria of the mean field game with an unbounded number of shocks where $\epsilon_n = L c^n$ with $L>0$ and $c\in(0,1)$. Finally, we illustrate the findings in an extension of the corruption detection model of Kolokoltsov and Malafeyev \cite{KolokoltsovCorruption2017} with random audits as common shocks.  
	
	\subsection{Related Literature}
	
	Starting with the seminal papers \cite{LasryJapanese2007} and \cite{HuangNCE2006}, mean field game theory has become a rich mathematical theory. Starting with games on a continuous state space with diffusive dynamics, nowadays there is a large variety of different mean field game models. For an overview of the general theory for games on a continuous state space with diffusive dynamics see 
	\cite{BensoussanMFG, CardaliaguetLectureNotes, CarmonaMFG2018, CarmonaMFGPartTwo2018}. Moreover, as mentioned above, applications in various fields have been analysed. These include \cite{ParisPrinceton2010,CarmonaSystemicRisk,NutzCompetitionRD} (economics/finance), \cite{KolokoltsovBotnet2016} (computer science), \cite{AurellEpidemics,doncel_gast_gaujal_2022_sir,laguzet2015individual} (epidemiology) and  \cite{GomesSocio2014, KolokoltsovCorruption2017} (sociology), among others.
	
	Mean field games that include common noise have first been studied in diffusion settings. Early works include \cite{CarmonaCommonNoise} and \cite{AhujaCommonNoise}; for a general overview we refer the reader to \cite{CarmonaMFGPartTwo2018} as well as the recent works \cite{Cardaliaguet2022First_Order_Common} and \cite{Lacker2023Convergence_Common}, which provide a good overview on the recent progress for models with diffusive dynamics. We highlight that the focus in this strand of literature is mainly on existence and uniqueness as well as on the convergence problem. 
	Results regarding the (numerical) computation of equilibria are only available for particular classes of mean field games, where equilibria depend in a simple way on the realization of the common noise. Up to the knowledge of the authors, there are two different classes where the computation is possible:
	On the one hand,  if the common noise events can only occur at a finite number of predefined time points and can take only finitely many values \cite{BelakMFG,dumitrescu2023energy}. On the other hand, in examples where the structural properties allow for a simple dependence of the equilibrium on the common noise, in which case  usually a guess-and-verify approach is used (\cite{CarmonaSystemicRisk}, \cite[Section 4.5]{CarmonaMFGPartTwo2018}, \cite{escribe_renewable_investment} and \cite{lavigne2023decarbonization}). In this paper we go beyond these two model classes by considering general models with a common noise that is driven by doubly stochastic Poisson processes.
	
	The focus in the present paper lies on mean field games with continuous time and finite state space. Games without common noise have been analysed in \cite{BayraktarMasterEquation, CarmonaExtended, CecchinProbabilistic2018, CecchinMasterEquation, DoncelPaper, GomesConti2013, GueantCongestion2015, Neumann2020} among others; see also §7.2 in \cite{CarmonaMFG2018}. 
	Games with common noise have been analysed in \cite{BayraktarFiniteStateCommonNoise,BayraktarFiniteStateCommonNoise_2, BelakMFG, BertucciRemarks, Cecchin_Vanishing_Common_Noise, DelarueFiniteCommonNoise}. 
	In contrast to continuous state space models (where additive noise is usually considered) there is no classical form of common noise. 
	Most authors consider a common source of randomness that only influences the dynamics of the players.
	This happens in an additive way \cite{DelarueFiniteCommonNoise}, through joint (random) state transitions at (random) times \cite{BertucciRemarks} or by including a common continuous-time Brownian motion into the dynamics inspired by population genetics \cite{BayraktarFiniteStateCommonNoise,BayraktarFiniteStateCommonNoise_2,Cecchin_Vanishing_Common_Noise}.
	Again all these papers focus on existence, uniqueness, convergence, equilibrium selection as well as the study of the master equation.
	A different route is followed in \cite{BelakMFG}, where the common noise can only occur at a finite number of predetermined deterministic times and can take finitely many values.
	First, the effect of the common noise in this setup is more general, it might yield to instantaneous jumps, changes of the effects of the actions on the state dynamics as well as changes of the rewards.
	Moreover, the authors derive a system of forward-backward ODE systems, which proves to be numerically tractable. We highlight that although we follow a similar route as in \cite{BelakMFG}, our common noise is much more general as we allow the common shocks to occur at any time.

	The remainder of this paper is organized as follows: We set up the mean field model with a finite number of shocks in Section~\ref{sec:model}. In Section~\ref{sec:opt_aggre_fb} we establish a verification theorem for the individual agent's optimization problem and compute the dynamics of the conditional aggregate distribution. Moreover we derive the forward-backward system and state the corresponding existence and uniqueness result. In Section~\ref{sec:approximation} we address the mean field model with an unbounded number of shocks and establish that equilibria of the game with $n$ shocks are approximate equilibria for these games. Section~\ref{sec:corruption} illustrates our approach in a corruption detection model. Finally, Appendix~\ref{appendix:measure_change} provides the formal construction of the state dynamics, and Appendix~\ref{appendix:existence} provides the proof of the existence and uniqueness result.

	\section{Mean Field Model} \label{sec:model}
	
	First, we provide an informal description of our mean field model setup including common noise realizations: 
	Each agent controls her own state $X_t$, which takes values in a finite state space, identical for all agents.
	In addition, the agents collectively face finitely many common shock events that occur sequentially at random times.
	The individual players' state dynamics are modeled as follows: If at time $t$ the agent is in state $i$, the agents' aggregate distribution across states is $M_t$, and the agent chooses action $\alpha_t$, then her individual state switches from $i$ to $j$ with intensity $Q^{ij}(t,Z_t,M_t,\alpha_t)$.
	Here $Z_t$ denotes the number of common shocks that have occurred up to time $t$; and $\lambda^k(t,M_t)$ is the intensity of shock $k+1$ given $k$ shocks have occurred already.
	Moreover, if there is a shock at time $t$, the agent's state is relocated from state $i$ to $J^i(t,M_{t-})$. Given these state dynamics, the agent's goal is to maximize the reward\footnote{For ease of notation, here and throughout the paper we write $X_t$ instead of $X_{t-}$, $M_t$ instead of $M_{t-}$, etc., whenever it does not make a difference.}
	\[
	\mathbb{E}^\alpha \left[ \int_0^T \psi^{X_t}(t,Z_t,M_t,\alpha_t) \de t + \Psi^{X_T} (Z_T,M_T) \right]
	\]
	over all admissible strategies $\alpha$, where $\psi^i$ and $\Psi^i$ represent, respectively, the running and the terminal reward in state $i$.\footnote{More generally $Q^{ij}$, $\psi^i$ and $\Psi^i$ may also depend on the \textit{path} of the common shock events; see below.}
	This determines the agents' strategy $\alpha$, \textit{given} a stochastic process $M$ for the underlying aggregate distribution.
	To close the loop, note that the aggregate distribution $M_t$ in fact derives from the agents' individual states, being the distribution of $X$ under $\mathbb{P}^\alpha$.
	This implies the equilibrium condition
	\[
	\mathbb{P}^\alpha(X_t=\cdot \; |Z_s, s\le t) = M_t\quad\text{for all }t\in[0,T].
	\]
	Note that, since common shocks affect all agents simultaneously, the aggregate distribution depends on the realization of common shock events. In particular, in contrast to mean field games without common shocks, the dynamics of the aggregate distribution $M_t$ are stochastic.
	
	\subsection{Mathematical Setup}
	
	We fix a probability space $(\Omega,\mathfrak{A},\mathbb{P})$, a finite time horizon $T>0$ and a maximum total number of common shocks $n\in\mathbb{N}$.
	We refer to $\mathbb{P}$ as the \textit{reference probability measure}.
	The individual agent's state space is represented by $\mathbb{S}=\{1,\ldots,S\}$ and her action space by  $\mathbb{A}\subseteq\mathbb{R}^A$.
	We assume that the functions
	\begin{align*}
		\lambda &:[0,T] \times \mathcal{P}(\mathbb{S}) \rightarrow (0,\infty)^n \\
		Q &: [0,T] \times \{0, 1, \ldots, n\} \times \mathcal{P}(\mathbb{S}) \times \mathbb{A} \rightarrow \mathbb{R}^{S \times S} \\
		\psi &: [0,T] \times \{0, 1, \ldots, n\} \times \mathcal{P}(\mathbb{S}) \times \mathbb{A} \rightarrow \mathbb{R}^S \\
		\Psi &: \{0,1, \ldots, n\} \times \mathcal{P}(\mathbb{S}) \rightarrow \mathbb{R}^S \\
		J &:[0,T] \times \mathcal{P}(\mathbb{S}) \rightarrow \mathbb{S}^\mathbb{S}
	\end{align*} 
	are bounded and Borel measurable where $\mathcal{P}(\mathbb{S})$ denotes the probability simplex on $\mathbb{S}$. Moreover, $Q(t,k,m,a)\in\mathbb{R}^{S\times S}$ is an intensity matrix for all $t \in [0,T]$, $k \in \{0,1, \ldots, n\}$, $m \in \mathcal{P}(\mathbb{S})$ and $a \in\mathbb{A}$.
	
	The dynamics of the common shocks process $Z$ are given by
	\begin{equation}
		\label{eq:model_Z_dynamics}
		\de Z_t = \sum_{k=1}^n \mathbb{I}_{\{Z_{t-}=k-1\}} \de N_t^k
	\end{equation}
	where $Z_0=0$ and $N^k$, $k\in\{1,\ldots,n\}$ are standard (unit-intensity) Poisson processes.
	We denote the filtration generated by common shock events by $\mathfrak{G}=(\mathfrak{G}_t)_{t\in[0,T]}$ where
	\begin{equation}
		\label{eq:def_filtration_g}
		\mathfrak{G}_t \defined \sigma(Z_s : s \le t), \quad t \in [0,T].
	\end{equation}
	The dynamics of the individual agent's state process $X$ are given by
	\begin{equation}
		\label{eq:model_X_dynamics}
		\begin{aligned}
			\de X_t &= \sum_{i,j \in \mathbb{S}:i \neq j} \sum_{k=0}^n \mathbb{I}_{\{X_{t-}=i\}} \mathbb{I}_{\{Z_{t-}=k\}} (j-i) \de N_t^{ikj} \\
			&\quad + \sum_{i\in \mathbb{S}} \sum_{k=1}^n \mathbb{I}_{\{X_{t-}=i\}} \mathbb{I}_{\{Z_{t-}=k-1\}} \left( J^i(t,M_{t-})-i \right) \de N^k_t
		\end{aligned}
	\end{equation}
	where $X_0$ is a fixed $\mathbb{S}$-valued random variable and $M$ is a fixed exogenous $\mathfrak{G}$-adapted process.
	Here $N^{ikj}$, $i,j \in \mathbb{S}, i \neq j$, $k \in \{0,1,\ldots, n\}$ are standard (unit-intensity) Poisson processes such that 
	\[
	X_0, \quad (N^{ikj})_{i,j \in \mathbb{S}, i \neq j, k \in \{0,1, \ldots, n\}} \quad \text{and} \quad (N^k)_{k \in \{1,\ldots,n\}}
	\]
	are independent under the reference probability $\mathbb{P}$.\footnote{Note that $N^k$ (single upper index) and $N^{ikj}$ (triple upper index) are distinct processes.}
	Thus the counting processes $N^k$, $k \in \{1, \ldots,n\}$ and $N^{ikj}$, $i,j \in \mathbb{S}, i \neq j$, $k \in \{0,1, \ldots, n\}$ trigger the jumps of the common shocks process $Z$ and the agent's state process $X$, respectively.
	Specifically, the process $X$ jumps on jumps of $N^k$ as specified in the relocation function $J$, and $X$ jumps from $i$ to $j$ on jumps of $N^{ikj}$ whenever the current number of shocks is $k$.
	Note that this is a pathwise construction, and that $Z$ and $X$ have unit jump intensities under the reference probability $\mathbb{P}$; these intensities will be transformed by a suitable change of measure below to match those in the informal model description above.
	The associated filtration $\mathfrak{F}=(\mathfrak{F}_t)_{t \in [0,T]}$ is given by 
	\[
	\mathfrak{F}_t \defined \sigma \left(X_0, N^{ikj}_s, N^l_s: s \in [0,t]; i,j \in \mathbb{S}, i \neq j; k \in \{0,1, \ldots,n\}; l \in \{1,\ldots,n\} \right), \ t \in [0,T].
	\]
	
	The agent's strategy $\alpha$ is required to be non-anticipative, i.e.\ the time-$t$ decision is required to be a functional of $t$, the common shock path $Z_{(\cdot\wedge t)-}$ and the path of her state process $X_{(\cdot\wedge t)-}$ up to immediately before time $t$. 
	Note that each function $\alpha:[0,T]\times \{0,1,\ldots,n\}^{[0,T]} \times\mathbb{S}^{[0,T]} \rightarrow \mathbb{A}$ canonically induces an $\mathbb{A}$-valued process $(\alpha_t)_{t\in[0,T]}$ via
	\[
	\alpha_t = \alpha\left(t,Z_{(\cdot\wedge t)-}, X_{(\cdot \wedge t)-}\right),\quad t\in[0,T].
	\]
	The set of \textit{admissible strategies} is thus given by
	\[
	\mathcal{A} \defined \left\{ \alpha: [0,T]\times \{0,1,\ldots,n\}^{[0,T]}  \times \mathbb{S}^{[0,T]} \rightarrow \mathbb{A}: (\alpha_t)_{t\in[0,T]}\ \text{is $\mathfrak{F}$-predictable} \right\}.
	\]
	Given the common shocks process $Z$ and a fixed exogenous aggregate distribution process $M$, the individual agent aims to maximize the reward functional
	\begin{equation}
		\label{eq:Optimzation}
		\mathbb{E}^\alpha \left[ \int_0^T \psi^{X_t} (t, Z_t, M_t, \alpha_t)\de t + \Psi^{X_T}(Z_T,M_T) \right]\quad\text{over all }\alpha\in\mathcal{A}.
	\end{equation}
	Here $\mathbb{E}^\alpha$ denotes expectation with respect to the equivalent probability measure $\mathbb{P}^\alpha$ defined formally in \eqref{eq:DefPAlpha}.
	The characteristic property of $\mathbb{P}^\alpha$ is that $N^k$ has $\mathbb{P}^\alpha$-intensity $\lambda^k=\{\lambda^k(t,M_t)\}_{t\in[0,T]}$ on $\{t:Z_{t}=k-1\}$ and, analogously, $ N^{ikj}$ has $\mathbb{P}^\alpha$-intensity $\lambda^{ikj}=\{Q^{ij}(t,k,M_t,\alpha_t)\}_{t\in[0,T]}$ on $\{t:Z_{t}=k\}$. Expressly, under $\mathbb{P}^\alpha$ the common shocks process $Z$ jumps from $k-1$ to $k$ with intensity $\lambda^k$ and the agent's state process $X$ jumps from $i$ to $j$ with intensity $\lambda^{ikj}$ when $k$ common shocks have occurred.
	The construction of $\mathbb{P}^\alpha$ is obtained by a well-known change-of-measure argument, see e.g.\ \cite{BelakMFG}; we refer to Section~\ref{sec:measure_change_bounded} in the Appendix for details and the explicit definition of $\mathbb{P}^\alpha$.
	
	\begin{definition}
		\label{def:MFE}
		A \textit{mean field equilibrium} is a pair $(\hat\alpha,\hat M)$ consisting of an admissible strategy $\hat\alpha\in\mathcal{A}$ and a $\mathcal{P}(\mathbb{S})$-valued, $\mathfrak{G}$-adapted process $\hat M$ such that
		\begin{align*}
			&\mathbb{E}^{\hat\alpha} \left[ \int_0^T \psi^{X_t} (t, Z_t,\hat M_t, \hat\alpha_t) \de t + \Psi^{X_T}(Z_T,\hat M_T) \right] \\ 
			&= \sup_{\alpha\in\mathcal{A}} \mathbb{E}^{\alpha} \left[ \int_0^T \psi^{X_t} (t, Z_t, \hat M_t, \alpha_t)\de t + \Psi^{X_T}(Z_T,\hat M_T) \right]
		\end{align*}
		while at the same time
		\[
		\mathbb{P}^{\hat\alpha}(X_t =\cdot \;|\mathfrak{G}_t) = \hat M_t\quad\text{for all }t\in[0,T].
		\]
	\end{definition}
	
	We emphasize that, since common noise events affect all agents simultaneously, in equilibrium the aggregate distribution $\hat M_t$ will depend on the realizations of all common shocks $Z_s$, $s\le t$, in a possibly non-Markovian way.

	\section{Optimization, Aggregation and Forward-Backward System}\label{sec:opt_aggre_fb}
	
	To solve the mean field equilibrium problem, we proceed in three classical steps:
	First, we solve for the individual agent's optimal strategy \textit{given} a stochastically evolving aggregate distribution (verification, see Section~\ref{sec:equiv_Markov}).
	Second, we identify conversely the stochastic dynamics of the aggregate distribution \textit{given} all agents implement the same feedback strategy (aggregation, see Section~\ref{sec:aggregation}).
	Third and finally, we bring together the verification and aggregation steps to formulate the mean field forward-backward system (see Section~\ref{sec:fb_system}).
	The challenge here is that, in contrast to mean field games without common noise, the evolution of the equilibrium aggregate distribution is stochastic.
	
	\subsection{Optimal Control Problem and Verification}\label{sec:equiv_Markov}
	
	The key mathematical step in addressing the individual agent's optimization problem is to find an equivalent reformulation in terms of a Markovian system.
	To achieve this, we define the $\mathbb{R}^n$-valued $\mathfrak{G}$-adapted process $U=(U_t)_{t\in[0,T]}$ via
	\begin{align}
		\label{eq:DefinitionU}
		U_t^k \defined \int_0^t (s+1) \mathbb{I}_{\{Z_{s-}=k-1\}}  \de N^k_s -1,\quad t\in[0,T],\ k\in\{1,\ldots,n\}.
	\end{align} 
	
	Thus $U_t$ is a vector such that, if shock number $k$ has occurred before $t$, $U_t^k$ records the point of time when it happened; and $U_t^k=-1$ else, i.e.\ if shock number $k$ has not yet occurred at time $t$.
	It is clear that $U_0 = u_0 \defined (-1,\ldots,-1)$ and that $U$ takes values in
	\begin{align*}
		\mathbb{U}&\defined \left\{ u\in \left( \{-1\} \cup [0,T]\right)^n: u_1\le u_2\le \ldots\le u_{k-1}, \right.\\
		&\qquad \left. u_k=\cdots=u_n=-1\ \text{for some }k\in\{1,\dots,n\} \right\}.
	\end{align*}
	In particular, the process $(t,U_t)_{t\in[0,T]}$ takes values in the space $\text{TS}$ defined as follows:
	\begin{definition} We set
		\[
		\text{TS} \defined \left\{ (t,u) \in [0,T] \times \mathbb{U} : \max_{k\in\{1,\dots,n\}}u_k\le t \right\}
		\]
		and we say that a function $f:\text{TS} \rightarrow \mathbb{R}^S$ is \textit{regular} if for all $u\in\mathbb{U}$ the function $f(\cdot,u)$ is absolutely continuous on $[\max_{k\in\{1,\dots,n\}}u_k,T]$.
	\end{definition}
	
	Since there is a one-to-one correspondence between $U_t$ and $(Z_s)_{s\in[0,t]}$,\footnote{In detail: For $s\le t$ we have $Z_s=0$ if and only if ($s < U_t^1$ or $U_t^1 = -1$); and $Z_s=k$ if and only if $U_t^{k}\le s$ and ($ U_t^{k+1}=-1$ or $U_t^{k+1} > s$) for all $k\in\{1,\ldots,n\}$.} it follows that
	\[
	\mathfrak{G}_t = \sigma(U_t),\quad t\in[0,T].
	\]
	In particular, every $\mathfrak{G}_t$-measurable random variable, such as the aggregate distribution $M_t$, can be represented as a measurable function of $U_t$.
	
	\begin{definition}
		Let $u\in\mathbb{U}$. We define $Z(u)$ as the $k \in \{0,1, \ldots, n\}$ such that $u^{l} \neq -1$ for all $l\le k$ and $u^{l}=-1$ for all $k+1 \le l \le n$.
		If $Z(u) < n$ and $t\ge u^{Z(u)}$, we define $\overrightarrow{(u,t)}\in\mathbb{U}$ by $\overrightarrow{(u,t)}^l = u^l$ for all $l\neq Z(u) +1$ and $\overrightarrow{(u,t)}^{Z(u)+1}=t$. 
		Moreover, if $Z(u)>0$ and $t=u^{Z(u)}$ we define $\overleftarrow{u}\in\mathbb{U}$ such that $\overleftarrow{u}^l = u^l$ for $l \neq Z(u)$ and $\overleftarrow{u}^{Z(u)} = -1$.
	\end{definition}
	
	We note that $Z(U_t)=Z_t$, i.e.\ $Z(U_t)$ represents the total number of shocks occurred up to and including time $t$. Moreover, we have $U_t =\overrightarrow{(U_{t-},t)}$ on $\{\Delta Z_t \neq 0\}$, i.e.\ $\overrightarrow{(U_{t-},t)}$ describes the vector $U_t$ if a shock occurs at time $t$. Analogously, $U_{t-} =\overleftarrow{U_t}$ on $\{\Delta Z_t \neq 0\}$, i.e.\  $\overleftarrow{U_t}$ describes the vector $U_{t-}$ if a shock occurs at time $t$.
	
	We now provide an equivalent Markovian reformulation of the agent's optimization problem \eqref{eq:Optimzation}, taking as given a fixed Markovian representation of the aggregate distribution process; i.e., a regular function $\mu: \text{TS} \rightarrow \mathcal{P}(\mathbb{S})$ such that $M_t = \mu(t,U_t)$ for $t\in[0,T]$. We take a classical dynamic programming approach.
	
	\begin{assumption}
		\label{assumption:H}
		There exists a Borel measurable function $h:\text{TS} \times \mathcal{P}(\mathbb{S}) \times \mathbb{R}^S \rightarrow \mathbb{A}^S$ such that for all $i \in \mathbb{S}$, $(t,u) \in \text{TS}$, $m \in \mathcal{P}(\mathbb{S})$ and $v \in \mathbb{R}^S$ we have
		\[
		h^i(t,u,m,v) \in \argmax_{a \in \mathbb{A}} \left\{
		\psi^i(t,Z(u),m,a) + \sum_{j \in \mathbb{S}} Q^{ij}(t,Z(u),m,a)v^j \right\}.		
		\]
	\end{assumption}
	
	For instance, Assumption~\ref{assumption:H} holds whenever $\mathbb{A}$ is compact and $Q$ and $\psi$ are continuous with respect to $a\in\mathbb{A}$. 
	Given a function $h$ satisfying Assumption~\ref{assumption:H} we define
	\[
	\hat{Q}: \text{TS} \times \mathcal{P}(\mathbb{S}) \times \mathbb{R}^S \rightarrow \mathbb{R}^{S \times S},\quad 
	\hat{Q}^{ij}(t,u,m,v) \defined Q^{ij}(t,Z(u),m,h^i(t,u,m,v)) 
	\]
	and
	\[
	\hat{\psi}:\text{TS} \times \mathcal{P}(\mathbb{S}) \times \mathbb{R}^S \rightarrow \mathbb{R}^S,\quad
	\hat\psi^i(t,u,m,v) \defined  \psi^i(t,Z(u),m,h^i(t,u,m,v)).
	\]
	With this notation, we can state the reduced dynamic programming equation associated with the individual agent's optimization problem.
	\begin{definition} Let $\mu: \text{TS} \rightarrow \mathcal{P}(\mathbb{S})$ be regular. 
		A regular function $v: \text{TS} \rightarrow \mathbb{R}^S$ is a solution of \eqref{eq:DP_mu} if it satisfies the ODE\footnote{All ODEs in this paper are taken in the sense of Carathéodory, see \cite[§1]{FilippovDiscontiODE1988}. Briefly, $x:[t_0,t_1] \rightarrow \mathbb{R}$ is a solution of $\dot{x}(t) = f(t,x(t))$, $x(t_0)=x_0$ in the sense of Carathéodory if $x$ is absolutely continuous and satisfies $x(t) = x_0 + \int_{t_0}^t f(s,x(s)) \de s$ for all $t \in [t_0, t_1]$.}
		\begin{equation}
			\label{eq:DP_mu}
			\tag{$\text{DP}_\mu$}
			\begin{split}
				&-\hat{\psi}^i(t,u,\mu(t,u),v(t,u)) \\
				& = \frac{\partial}{\partial t} v^{i}(t,u) + \sum_{j \in \mathbb{S}} \hat{Q}^{ij}(t,u,\mu(t,u), v(t,u)) v^j(t,u)  \\
				& \quad +  \mathbb{I}_{\{Z(u)<n\}} \lambda^{Z(u)+1}(t, \mu(t,u)) \left(v^{J^i(t,\mu(t-,u))}\left(t,\overrightarrow{(u,t)}\right) - v^i(t,u) \right)  \\
			\end{split}
		\end{equation}
		for all $i \in \mathbb{S}$ and $(t,u) \in \text{TS}$ subject to the terminal condition
		\[v(T,u) = \Psi(Z(u),\mu(T,u)).\]
	\end{definition}
	
	\begin{remark}
		\label{remark:ODEv}
		The system \eqref{eq:DP_mu} is a family of ODEs parameterized by $u\in\mathbb{U}$.
		We show in Lemma~\ref{lemma:appendix_existence_v} that, under natural continuity conditions, it admits a unique solution.
		Moreover, this solution can be found recursively:
		For $u\in\mathbb{U}$ with $Z(u)=n$ we face a classical ODE that does not depend on $v(\cdot,\tilde u)$ for any other parameter $\tilde u\in\mathbb{U}$.
		Having solved for all $v(\cdot,u)$ with $Z(u)=n$, in the next step the equations for $v(\cdot,u)$ with $Z(u)=n-1$ are classical ODEs depending on the (known) functions $v(\cdot,\bar u)$ with $Z(\bar u)=n$.
		Since one can show that $v(\cdot,u)$ is continuous in $u^n$, i.e.\ the time of the last shock (see the proof of Lemma~\ref{lemma:appendix_existence_v}), the right-hand sides of these ODEs are again continuous in $t$ and we obtain a solution to the ODEs.
		Proceeding in this fashion, the system can be solved recursively. \remEnde
	\end{remark}
	
	\begin{theorem}[verification]
		\label{thm:Verification}
		Let $\mu: \text{TS} \rightarrow \mathcal{P}(\mathbb{S})$ be a regular function and suppose $v$ is a solution of \eqref{eq:DP_mu}. Then $v$ is the agent's value function for the optimization problem, i.e. 
		\[
		\sum_{i \in \mathbb{S}} \mathbb{P}(X_0=i) v^i(0,u_0) = \sup_{\alpha \in \mathcal{A}} \mathbb{E}^\alpha \left[ \int_0^T \psi^{X_t} (t, Z_t, M_t, \alpha_t) \de t + \Psi^{X_T}(Z_T,M_T) \right]
		\]
		where $M_t= \mu(t,U_t)$ for $t\in[0,T]$, and an optimal control is given by 
		\begin{equation}\label{eq:strategy}
			\hat\alpha(t, Z_{(\cdot \wedge t)-},X_{(\cdot \wedge t)-}) = h^{X_{t-}} (t, U_{t-}, \mu(t, U_{t-}), v(t,U_{t-})).
		\end{equation}
	\end{theorem}
	
	The strategy in Theorem~\ref{thm:Verification} is Markov with respect to the agent's state and $U_t$. More generally, such strategies are represented via the feedback form
	\begin{equation}
		\label{eq:Definition_Feedback_Strategy}
		\alpha(t,Z_{(\cdot\wedge t)_-},X_{(\cdot \wedge t)_-}) = \nu(t,U_{t-},X_{t-}),\quad t\in[0,T]
	\end{equation}
	where $\nu: \text{TS} \times \mathbb{S} \rightarrow \mathbb{A}$ is Borel measurable.
	Note that these feedback strategies need not be Markov with respect to common shocks, as $U_t$ captures the information of the entire path $(Z_s)_{s\in[0,t]}$; in particular, it is possible for the agent to condition her strategy on the exact timing of the shocks.
	
	\begin{proof}[Proof of Theorem~\ref{thm:Verification}]
		We generalize the standard verification argument used in \cite{GomesConti2013}, \cite[Ch. 7.2]{CarmonaMFG2018}, \cite{BelakMFG}, among others, to take into account the additional state process $U$ representing the common shock events.
		Noting that $N^{ikj}$ and $U$ do almost surely not jump simultaneously and that $v$ is continuous in $t$ we obtain from It\^{o}'s Lemma
		\begin{align*}
			v^{X_T}(T,U_T)
			&= v^{X_0} (0,u_0) \\
			&\quad + \sum_{i \in \mathbb{S}} \int_0^T \mathbb{I}_{\{X_{s-}=i\}} \frac{\partial}{\partial t} v^i(s,U_{s-}) \de s \\
			&\quad + \sum_{i,j \in \mathbb{S}: i \neq j} \sum_{k=0}^n \int_0^T \mathbb{I}_{\{X_{s-}=i\}} \mathbb{I}_{\{Z_{s-}=k\}} (v^j(s,U_s)-v^i(s,U_{s-}))  \de N_s^{ikj} \\
			&\quad + \sum_{i \in \mathbb{S}} \sum_{l=1}^n \int_0^T \mathbb{I}_{\{X_{s-}=i\}}   \mathbb{I}_{\{Z_{s-}=l-1\}} (v^{J^i(s,M_{s-})}(s,U_s)- v^i(s,U_{s-})) \de N_s^l \\
			&= v^{X_0} (0,u_0) \\
			&\quad + \sum_{i \in \mathbb{S}} \int_0^T \mathbb{I}_{\{X_{s-}=i\}}  \frac{\partial}{\partial t} v^i(s,U_{s-}) \de s \\
			&\quad + \sum_{i,j \in \mathbb{S}: i \neq j} \sum_{k =0}^n \left(  \int_0^T \mathbb{I}_{\{X_{s-}=i\}} \mathbb{I}_{\{Z(U_{s-})=k\}} (v^j(s,U_{s-})-v^i(s,U_{s-})) \de \overline{N}_s^{ikj} \right. \\
			&\quad +  \int_0^T \mathbb{I}_{\{X_{s-}=i\}} \mathbb{I}_{\{Z(U_{s-})=k\}}   (v^j(s,U_{s-})-v^i(s,U_{s-}))   Q^{ij}(s,k,M_s, \alpha_s)  \de s \Bigg) \\
			&\quad + \sum_{i \in \mathbb{S}} \sum_{l=1}^n \left(  \int_0^T \mathbb{I}_{\{X_{s-}=i\}} \mathbb{I}_{\{Z(U_{s-})=l-1 \}} \left(v^{J^i(s,M_{s-})}\left(s,\overrightarrow{(U_{s-},s)}\right)- v^i(s,U_{s-})\right) \de \overline{N}_s^l \right. \\
			&\quad +  \left. \int_0^T \mathbb{I}_{\{X_{s-}=i\}} \mathbb{I}_{\{Z(U_{s-})= l-1\}} \left(v^{J^i(s,M_{s-})}\left(s,\overrightarrow{(U_{s-},s)}\right)- v^i(s,U_{s-})\right) \lambda^l(s,M_s) \de s \right) \\
			&= v^{X_0}(0,u_0) \\
			&\quad + \sum_{i,j \in \mathbb{S}: i \neq j} \sum_{k=0}^n \int_0^T \mathbb{I}_{\{X_{s-}=i\}} \mathbb{I}_{\{Z(U_{s-})=k\}} (v^j(s,U_{s-})-v^i(s,U_{s-})) \de \overline{N}_s^{ikj} \\
			&\quad + \sum_{i \in \mathbb{S}} \sum_{l=1}^n \int_0^T  \mathbb{I}_{\{X_{s-}=i\}} \mathbb{I}_{\{Z(U_{s-})=l-1\}} \left(v^{J^i(s,M_{s-})}\left(s,\overrightarrow{(U_{s-},s)}\right)- v^i(s,U_{s-})\right) \de \overline{N}_s^l \\
			&\quad + \sum_{i \in \mathbb{S}} \int_0^T \mathbb{I}_{\{X_{s-}=i\}}  \Bigg( \frac{\partial}{\partial t} v^{i}(s,U_{s-}) \\
			&\quad + 
			\mathbb{I}_{\{Z(U_{s-})<n\}} \left(v^{J^i(s,M_{s-})}\left(s,\overrightarrow{(U_{s-},s)}\right)- v^i(s,U_{s-})\right)  \lambda^{Z(U_{s-})+1}(s,M_s)   \\
			&\quad + \sum_{j \in \mathbb{S}}   Q^{ij}(s,Z(U_{s-}),M_s, \alpha_s) v^j(s,U_{s-})  \Bigg) \de s. 
		\end{align*}
		where the compensated Poisson processes $\overline{N}_s^{ikj}$ and $\overline{N}_s^l$ are $L^2$-martingales. Since $v$ and $Q$ are bounded, the corresponding stochastic integrals are martingales.
		Thus we obtain
		\begin{align*}
			&\mathbb{E}^\alpha \left[ v^{X_0}(0,U_0) \right] \\
			&= \mathbb{E}^\alpha[v^{X_T}(U_T)]  - \mathbb{E}^\alpha \Bigg[ \sum_{i \in \mathbb{S}} \int_0^T \mathbb{I}_{\{X_{s-}=i\}}  \Bigg( \frac{\partial}{\partial t} v^{i}(s,U_{s-}) \\
			&\qquad +
			\mathbb{I}_{\{Z(U_{s-})<n\}} \left(v^{J^i(s,M_{s-})}\left(s,\overrightarrow{(U_{s-},s)}\right)- v^i(s,U_{s-})\right)  \lambda^{Z(U_{s-})+1}(s,M_s)   \\
			&\qquad + \sum_{j \in \mathbb{S}}  Q^{ij}(s,Z(U_{s-}),M_s, \alpha_s) v^j(s,U_{s-})  \Bigg) \de s  \Bigg] \\
			&\ge \mathbb{E}^\alpha \Bigg[ \Psi^{X_T}(Z(U_T),M_T) + \int_0^T \psi^{X_{s-}}(s,Z(U_{s-}), M_{s-},\alpha_s) \de s \Bigg].\\
			&= \mathbb{E}^\alpha \left[\psi^{X_T}(Z_T,M_T) + \int_0^T \psi^{X_s}(s,Z_s,M_s, \alpha_s) \de s \right]
		\end{align*}
		with equality in the special case $\alpha=\hat\alpha$ as in \eqref{eq:strategy}.
		On the other hand, since $\mathbb{P}$ and $\mathbb{P}^\alpha$ coincide on $\mathfrak{F}(0)$ and $U_0=u_0$ a.s., it is clear that
		\[
		\sum_{i \in \mathbb{S}} \mathbb{P}(X_0=i) v^i(0,u_0) = \mathbb{E}\left[ v^{X_0}(0,U_0)\right] = \mathbb{E}^\alpha \left[ v^{X_0}(0,U_0) \right]
		\]
		and the proof is complete.
	\end{proof}
	
	\subsection{Aggregation}\label{sec:aggregation}
	
	In this section we determine the implied aggregate distribution resulting from a given feedback strategy via the associated Kolmogorov forward equation.
	As above, we write $\alpha_t = \alpha(t, U_{t-}, X_{t-})$, $t\in[0,T]$, whenever $\alpha$ is a feedback strategy as in \eqref{eq:Definition_Feedback_Strategy}.
	
	\begin{lemma}
		\label{Lemma:DistGivenMarkov}
		Let $M$ be a $\mathfrak{G}$-adapted process. If $\alpha: \text{TS} \times \mathbb{S}  \rightarrow \mathbb{A}$ is a feedback strategy, then the associated conditional aggregate distribution $\mathbb{P}^\alpha(X_t\in\cdot\,|\mathfrak{G}_T)$ satisfies
		\begin{align*}
			&\mathbb{P}^\alpha(X_t = l|\mathfrak{G}_T)\\
			&= \mathbb{P}(X_0 = l) + \sum_{i \in \mathbb{S}} \int_0^t \mathbb{P}^\alpha(X_{s-} = i|\mathfrak{G}_T)  Q^{il}(s,Z(U_{s-}),M_s, \alpha(s,U_{s-},i)) \de s \\
			&\quad + \sum_{k=1}^n \mathbb{I}_{\{Z(U_t) \ge k\}} \left( \sum_{i \in \mathbb{S}} \mathbb{I}_{\left\{J^i\left(U_t^k,M_{U_t^k-}\right) = l\right\}} \mathbb{P}^\alpha(X_{U_t^k-} = i|\mathfrak{G}_T) - \mathbb{P}^\alpha(X_{U_t^k-} =l| \mathfrak{G}_T) \right)
		\end{align*} for all $l \in \mathbb{S}$.
	\end{lemma}
	
	\begin{proof}
		Let $b \in \mathbb{R}^S$ be arbitrary. It\^{o}'s lemma yields
		\begin{align*}
			b^{X_t} &= b^{X_0} + \sum_{i,j \in \mathbb{S}:i \neq j } \sum_{k=0}^n \int_0^t \mathbb{I}_{\{X_{s-}=i\}} \mathbb{I}_{\{Z(U_{s-}) = k\}} (b^j-b^i) \de N_s^{ikj} \\
			&\quad + \sum_{i \in \mathbb{S}} \sum_{k=1}^n \int_0^t \mathbb{I}_{\{X_{s-}=i\}} \mathbb{I}_{\{Z(U_{s-}) = k-1\}} \left( b^{J^i(s,M_{s-})}-b^i \right)\de N_s^k \\
			&= b^{X_0} + \sum_{i,j \in \mathbb{S}:i \neq j } \sum_{k=0}^n \int_0^t \mathbb{I}_{\{X_{s-}=i\}} \mathbb{I}_{\{Z(U_{s-})=k\}} (b^j-b^i) \de \overline{N}_s^{ikj} \\
			&\quad + \sum_{i,j \in \mathbb{S}: i \neq j} \sum_{k=0}^n \int_0^t \mathbb{I}_{\{X_{s-}=i\}} \mathbb{I}_{\{Z(U_{s-})=k\}} (b^j-b^i) Q^{ij}(s,k,M_s, \alpha(s,U_{s-},i)) \de s \\
			&\quad + \sum_{i \in \mathbb{S}} \sum_{k=1}^n \mathbb{I}_{\{Z(U_t) \ge k\}} \mathbb{I}_{\left\{X_{U_t^k-} = i\right\}} \left(b^{J^i\left(U_t^k, M_{U_t^k-}\right)} -b^i\right).
		\end{align*}
		Taking conditional expectations and using Fubini's theorem we obtain 
		\begin{align*}
			&\mathbb{E}^\alpha\left[b^{X_t}|\mathfrak{G}_T\right]  - \mathbb{E}^\alpha \left[ b^{X_0} | \mathfrak{G}_T \right] \\
			&= \sum_{i,j \in \mathbb{S}:i \neq j } \sum_{k=0}^n \mathbb{E}^\alpha \left[ \int_0^t \mathbb{I}_{\{X_{s-}=i\}} \mathbb{I}_{\{Z(U_{s-})=k\}} (b^j-b^i) \de \overline{N}_s^{ikj} | \mathfrak{G}_T\right] \\
			&\quad + \sum_{i,j \in \mathbb{S}:i \neq j } \int_0^t \mathbb{E}^\alpha \left[ \mathbb{I}_{\{X_{s-}=i\}} (b^j-b^i) Q^{ij}(s,Z(U_{s-}),M_s, \alpha(s,U_{s-},i))|\mathfrak{G}_T \right] \de s \\
			&\quad + \sum_{i \in \mathbb{S}} \sum_{k=1}^n \mathbb{E}^\alpha \left[ \mathbb{I}_{\{Z(U_t) \ge k\}} \mathbb{I}_{\left\{X_{U_t^k-} = i\right\}} \left( b^{J^i\left(U_t^k, M_{U_t^k-}\right)} - b^i \right) |\mathfrak{G}_T\right].
		\end{align*}
		We next show that
		\[
		\mathbb{E}^\alpha \left[ \int_0^t \mathbb{I}_{\{X_{s-}=i\}} \mathbb{I}_{\{Z(U_{s-})=k\}} (b^j-b^i) \de \overline{N}_s^{ikj} | \mathfrak{G}_T \right] = 0
		\] for all $i, j \in \mathbb{S}$ with $i \neq j$ and all $k \in \{0,1, \ldots, n\}.$
		We denote paths of $(N^k)_{k \in \{1,\ldots, n\}}$ by $\omega_z$ and paths of $(X_0,(N^{ikj})_{i,j \in \mathbb{S}, i \neq j; k \in \{0,1, \ldots, n\}})$ by $\omega_x$.
		Recall that by Remark~\ref{rem:change_of_measure_two_step}, the change of measure is a product $\frac{\de \mathbb{P}^\alpha}{\de \mathbb{P}} = \prescript{}{Z}{\Theta^\alpha} \cdot \prescript{}{X}{\Theta^\alpha}$ consisting of the $\sigma(N^k, k \in \{1, \ldots, n\})$-measurable term $\prescript{}{Z}{\Theta^\alpha}$ and the term $\prescript{}{X}{\Theta^\alpha}$, which for each path of $Z$ describes a change of measure for the processes $\overline{N}^{ikj}$. 
		
		Let $G \in \mathfrak{G}$ be arbitrary. Then we obtain using Fubini's theorem
		\begin{align*}
			&\mathbb{E}^\alpha \left[ \mathbb{I}_G \int_0^t \mathbb{I}_{\{X_{s-} = i\}} \mathbb{I}_{\{Z(U_{s-})=k\}} (b^j-b^i) \de \overline{N}_s^{ikj} \right] \\
			&= \mathbb{E} \left[ \prescript{}{Z}{\Theta}^\alpha \prescript{}{X}{\Theta}^\alpha \mathbb{I}_G \int_0^t \mathbb{I}_{\{X_{s-} = i\}} \mathbb{I}_{\{Z(U_{s-})=k\}} (b^j-b^i) \de \overline{N}_s^{ikj} \right] \\
			&= \int \int \prescript{}{Z}{\Theta}^\alpha (\omega_z) \prescript{}{X}{\Theta}^\alpha (\omega_z, \omega_x)\mathbb{I}_G (\omega_z) \int_0^t \mathbb{I}_{\{X_{s-}(\omega_z,\omega_x) = i\}} \mathbb{I}_{\{Z(U_{s-})(\omega_z)=k\}} (b^j-b^i) \de \overline{N}_s^{ikj} (\omega_z, \omega_x) \\
			&\qquad \mathbb{P}^{(X_0, (N^{ikj})_{i,j \in \mathbb{S}, i \neq j; k \in \{0,1, \ldots, n\}})} (\de \omega_x) \mathbb{P}^{(N^k)_{k \in \{1, \ldots, n\}}}(\de \omega_z) \\
			&= \int \prescript{}{Z}{\Theta}^\alpha (\omega_z) \mathbb{I}_G (\omega_z) \int \prescript{}{X}{\Theta}^\alpha (\omega_z, \omega_x) \int_0^t \mathbb{I}_{\{X_{s-}(\omega_z,\omega_x) = i\}} \mathbb{I}_{\{Z(U_{s-})(\omega_z)=k\}} (b^j-b^i) \de \overline{N}_s^{ikj} (\omega_z, \omega_x) \\
			&\qquad \mathbb{P}^{(X_0,  (N^{ikj})_{i,j \in \mathbb{S}, i \neq j; k \in \{0,1, \ldots, n\}})} (\de \omega_x) \mathbb{P}^{(N^k)_{k \in \{1, \ldots, n\}}}(\de \omega_Z)
		\end{align*}
		Now, relying on the same argument as in Proposition~\ref{Thm:ChangeOfMeasure},  for any fixed path $\omega_Z$ the process $\overline{N}_s^{ikj}$ is a martingale under the probability measure with $\mathbb{P}$-density $\prescript{}{X}{\Theta^\alpha}$. Hence, the inner integral is a martingale and we conclude that the outer integral is zero, whence
		\[
		\mathbb{E}^\alpha \left[ \int_0^t \mathbb{I}_{\{X_{s-}=i\}} \mathbb{I}_{\{Z(U_{s-})=k\}} (b^j-b^i) \de \overline{N}_s^{ikj} | \mathfrak{G}_T \right] = 0.
		\]
		Using this as well as the fact that we pull out know factors from the conditional expectation we obtain
		\begin{align*}
			&\mathbb{E}^\alpha\left[b^{X_t}|\mathfrak{G}_T\right]  - \mathbb{E}\left[b^{X_0} \right] \\
			&=  \sum_{i,j \in \mathbb{S}:i \neq j } \int_0^t \mathbb{E}^\alpha \left[ \mathbb{I}_{\{X_{s-}=i\}} |\mathfrak{G}_T \right] (b^j-b^i) Q^{ij}(s,Z(U_{s-}),M_s, \alpha(s,U_{s-},i)) \de s \\
			&\quad + \sum_{i \in \mathbb{S}} \sum_{k=1}^n \mathbb{E}^\alpha \left[ \mathbb{I}_{\left\{X_{U_t^k-} = i\right\}} |\mathfrak{G}_T\right] \mathbb{I}_{\{Z(U_t) \ge k\}} \left( b^{J^i\left(U_t^k, M_{U_t^k-}\right)} -b^i \right).
		\end{align*}
		Choosing $b=e_l$ it follows that
		\begin{align*}
			&\mathbb{P}^\alpha(X_t = l|\mathfrak{G}_T)\\
			&= \mathbb{P}(X_0=l) \\
			&\quad + \sum_{i,j\in \mathbb{S}: i \neq j}  \int_0^t \mathbb{P}^\alpha(X_{s-} = i |\mathfrak{G}_T)  \left( \mathbb{I}_{\{l =j\}} - \mathbb{I}_{\{l=i\}} \right) Q^{ij}(s,Z(U_{s-}),M_s, \alpha(s,U_{s-},i)) \de s \\
			&\quad  + \sum_{i \in \mathbb{S}} \sum_{k=1}^n \mathbb{P}^\alpha \left(X_{U_t^k-} = i|\mathfrak{G}_T\right) \mathbb{I}_{\{Z(U_t) \ge k\}} \left( \mathbb{I}_{\left\{J^i\left(U_t^k,M_{U_t^k-}\right)=l\right\}} - \mathbb{I}_{\{i=l\}} \right) \\
			&= \mathbb{P}(X_0=l) + \sum_{i \neq l} \mathbb{P}^\alpha(X_{s-} = i|\mathfrak{G}_T)  Q^{il}(s,Z(U_{s-}),M_s, \alpha(s,U_{s-},i)) \de s \\
			&\quad - \sum_{j \neq l} \mathbb{P}^\alpha(X_{s-} = l|\mathfrak{G}_T) Q^{lj}(s,Z(U_{s-}),M_s, \alpha(s,U_{s-},l)) \de s \\
			&\quad + \sum_{k=1}^n \mathbb{I}_{\{Z(U_t) \ge k\}} \left( \sum_{i \in \mathbb{S}} \mathbb{I}_{\left\{J^i\left(U_t^k,M_{U_t^k-}\right) =l\right\}} \mathbb{P}^\alpha(X_{U_t^k-} = i|\mathfrak{G}_T) - \mathbb{P}^\alpha(X_{U_t^k-} = l|\mathfrak{G}_T) \right) \\
			&= \mathbb{P}(X_0=l) + \sum_{i \in \mathbb{S}} \int_0^t \mathbb{P}^\alpha(X_{s-} = i|\mathfrak{G}_T) Q^{il}(s,Z(U_{s-}),M_s, \alpha(s,U_{s-},i)) \de s \\
			&\quad + \sum_{k=1}^n \mathbb{I}_{\{Z(U_t) \ge k\}} \left( \sum_{i \in \mathbb{S}} \mathbb{I}_{\left\{J^i\left(U_t^k,M_{U_t^k-}\right) =l\right\}} \mathbb{P}^\alpha(X_{U_t^k-} = i|\mathfrak{G}_T) - \mathbb{P}^\alpha(X_{U_t^k-} = l|\mathfrak{G}_T) \right),
		\end{align*} which is the desired claim.
	\end{proof} 
	
	\begin{theorem}[conditional aggregate distribution]
		\label{thm:DistributionGivenMarkovStrat}
		Let $\mu:\text{TS} \rightarrow \mathbb{P}(\mathbb{S})$ be a regular function and $M_t=\mu(t,U_t)$, $t \in [0,T]$. Moreover, let $\alpha:\text{TS} \times \mathbb{S} \rightarrow \mathbb{A}$ be a feedback strategy. Then
		\[
		\mathbb{P}^\alpha(X_t=j|\mathfrak{G}_T) = \mathbb{P}^\alpha(X_t=j|U_t)\quad\text{for }t\in[0,T],\ j\in\mathbb{S}.
		\] 
		The dynamics of the conditional aggregate distribution are given by
		\begin{align*}
			\de \mathbb{P}^\alpha(X_t=j|U_{t}) &= \sum_{i \in \mathbb{S}} \mathbb{P}^\alpha(X_{t}=i|U_{t}) Q^{ij}(t,Z(U_{t}),M_t,\alpha(t,U_{t},i)) \de t\quad \text{on }\Delta U_t=0
		\end{align*} 
		and 
		\begin{align*}
			\mathbb{P}^\alpha(X_t=j|U_t) = \sum_{i \in \mathbb{S}} \mathbb{I}_{\left\{ J^i \left( t, M_{t-} \right) = j \right\}}\mathbb{P}^\alpha(X_t=i|U_{t-})\quad\text{on }\Delta U_t\neq 0
		\end{align*}
		with the initial condition
		\[
		\mathbb{P}^\alpha(X_0 = j|U_0) = \mathbb{P}(X_0=j).
		\]
	\end{theorem}
	
	\begin{proof}
		Since $\mathfrak{G}_T=\sigma(\mathfrak{G}_t\cup\mathcal{Z}_t)$ where $\mathcal{Z}_t=\sigma(Z_s:s>t)$ and $X_t$ is independent of $\mathcal{Z}_t$ given $Z_t$ we obtain by \cite[Prop.  6.6]{Kallenberg2002Foundations} that
		\[
		\mathbb{P}^\alpha(X_t=j|\mathfrak{G}_T) = \mathbb{P}^\alpha(X_t=j|\mathfrak{G}_t) = \mathbb{P}^\alpha(X_t=j|U_t),
		\] 
		which immediately yields for fixed $t$ the first representation for the case $\Delta U_t =0$.
		The representation for the case $\Delta U_t \neq 0$ immediately follows by comparing the formulae from Lemma \ref{Lemma:DistGivenMarkov} for $t=U_t^k$ and $s=U_t^k-$.
	\end{proof}
	
	\subsection{Forward-Backward System}\label{sec:fb_system}
	
	For a fixed regular function $\mu: \text{TS} \rightarrow \mathcal{P}(\mathbb{S})$, we denote by $\hat{\mathbb{P}}$ the distribution of $X_t$ given that the agent chooses the optimal strategy $\hat{\alpha}$ given in Theorem \ref{thm:Verification}.
	For $\hat{\alpha}$ to be an equilibrium strategy, it is necessary that $\hat{\mathbb{P}}(X_t =\cdot \;|U_{t}) = \mu (t,U_{t})$, i.e.\ the agent's prediction of the other agents' distribution is correct ex post.
	Combining Theorem~\ref{thm:Verification} with Theorem~\ref{thm:DistributionGivenMarkovStrat} it follows that we obtain a mean field equilibrium from the solution of the forward-backward equilibrium system stated below:
	
	\begin{definition}
		A pair $(\mu,v)$ of regular functions $	\mu :\text{TS} \rightarrow \mathcal{P}(\mathbb{S})$ and $v:\text{TS} \rightarrow \mathbb{R}^S$ is a solution of the \textit{equilibrium system}  if 
		\begin{align}
			\label{eq:FB1} \tag{E1}
			\begin{split}
				\frac{\partial}{\partial t} v^{i}(t,u)
				& =  - \hat{\psi}^i(t,u,\mu(t,u),v(t,u)) - \sum_{j \in \mathbb{S}} \hat{Q}^{ij}(t,u,\mu(t,u), v(t,u)) v^j(t,u)  \\
				& \quad - \mathbb{I}_{\{Z(u)<n\}} \lambda^{Z(u)+1} (t,\mu(t,u)) \left(v^{J^i(t, \mu(t-,u))}\left(t,\overrightarrow{(u,t)}\right) - v^i(t,u) \right)  
			\end{split} \\ \label{eq:FB2} \tag{E2}
			\frac{\partial}{\partial t} \mu^{i}(t,u) &= \sum_{j \in \mathbb{S}} \mu^j(t,u) \hat{Q}^{ji}(t,u,\mu(t,u),v(t,u))
		\end{align}
		for all $i \in \mathbb{S}$, $(t,u) \in \text{TS}$ and $m \in \mathcal{P}(\mathbb{S})$ subject to the initial and terminal conditions
		\begin{align}
			\label{eq:FB3} \tag{E3}
			v(T,u) &= \Psi(Z(u),\mu(T,u)) \\ \label{eq:FB4} \tag{E4}
			\mu(0,u_0) &= m_0
		\end{align}
		as well as the consistency conditions
		\begin{align}
			\label{eq:FB5} \tag{E5}
			\mu(t_u,u) = \mu\left(t_u,\overleftarrow{u}\right) \left( \mathbb{I}_{\{J^i(t_u ,\mu(t_u-,u)) = j\}} \right)_{i,j \in \mathbb{S}}\quad\text{for }t_u=u^{Z(u)},\ Z(u)\ge 1.
		\end{align}
	\end{definition}
	
	\begin{remark}
		\label{remark:ODEmu}
		We recall from Remark~\ref{remark:ODEv} that, given a regular function $\mu:\text{TS}\rightarrow\mathcal{P}(\mathbb{S})$, the backward equations \eqref{eq:FB1} and \eqref{eq:FB3} can be solved iteratively for $v$, starting from $u\in\mathbb{U}$ with $Z(u)=n$.
		Analogously, given a regular function $v:\text{TS}\rightarrow \mathbb{R}^{\mathbb{S}}$ under natural continuity conditions the forward equations \eqref{eq:FB2}, \eqref{eq:FB4} and \eqref{eq:FB5} can be solved recursively for $\mu$ starting from $u\in\mathbb{U}$ with $Z(u)=0$ (see Lemma~\ref{lemma:appendix_existence_mu}). \remEnde
	\end{remark}
	
	Under standard Lipschitz continuity assumptions we obtain existence and uniqueness of solutions to the system \eqref{eq:FB1}-\eqref{eq:FB5} on a short time horizon. 
	
	\begin{theorem}[existence and uniqueness for the equilibrium system]
		\label{thm:ex_unique_main_text}
		Assume that the functions $\hat{\psi}$ and $\hat{Q}$ are Lipschitz in $m$ and $v$, the functions $\Psi$ and $\lambda$ are Lipschitz in $m$ and that the relocation function $J^i$ is constant in $m$. Then for $T>0$ sufficiently small, there is a unique solution of \eqref{eq:FB1}-\eqref{eq:FB5}.
	\end{theorem}
	
	As usual, the proof (which is deferred to Appendix~\ref{appendix:existence}) is based on Banach's fixed point theorem. Hence, as a by-product, we obtain that equilibria on a short-time horizon can be computed through a fixed point iteration.
	
	\begin{remark} 
		It is possible to formulate sufficient conditions for the Lipschitz continuity of $\hat{Q}$ and $\hat{\psi}$ in terms of model primitives, see  \cite[Proposition 1]{GomesConti2013} and \cite[Lemma 7]{CecchinProbabilistic2018}. \remEnde
	\end{remark}

	\section{Approximate Equilibria for Infinitely Many Common Shocks}
	\label{sec:approximation}
	
	The approach discussed above does not directly generalize to mean field games with infinitely many shocks.
	However, the model with up to $n$ shocks represents an approximation of the model with an unbounded number of shocks. 
	More precisely, in this section we demonstrate that equilibria of the $n$-shock mean field game are $\epsilon_n$-equilibria of the game with an unbounded number of shocks, where $\epsilon_n = L c^n$ for some $c\in(0,1)$ and $L>0$.
	As a first step we formally extend our model setup to allow for an unbounded number of shocks.\\
	
	Let $(\Omega, \mathfrak{A}, \mathbb{P})$ now be a probability space that supports infinitely many standard Poisson processes $N^{ikj}$ for $i,j \in \mathbb{S}$ with $i \neq j$, $k \in \mathbb{N}_0$, standard Poisson processes $N^k$ for all $k \in \mathbb{N}$ and an $\mathbb{S}$-valued random variable such that $X_0$, $(N^{ikj})_{i,j \in \mathbb{S}, i \neq j, k \in \mathbb{N}_0}$ and $(N^k)_{k \in \mathbb{N}}$ are independent. Define the full filtration via
	\[
	\mathfrak{F}_t \defined \sigma \left( X_0, N_s^{ikj}, N^l_s : s \in [0,t]; i, j \in \mathbb{S}, i \neq j; k \in \mathbb{N}_0; l \in \mathbb{N} \right).
	\]
	The probability space in this section contains all elements introduced in the last section. 
	Hence, we can immediately define the shock process $Z^n$ via $Z^n_0 =0$ and \eqref{eq:model_Z_dynamics} and the state process $X^n$ via \eqref{eq:model_X_dynamics}, i.e.\ as in the setting with at most $n$ shocks.
	Similarly, we denote by $\mathfrak{G}^n$ the filtration generated by first $n$ common shock events (i.e.\ the process $Z^n$) and by $\mathcal{A}^n$ the set of admissible strategies for the game with $n$ shocks.
	Finally, we denote by $\Theta^{\alpha,n}$ the change of measure for the $n$ shock game given a strategy $\alpha \in \mathcal{A}^n$ and by $\mathbb{P}^{\alpha, n}$ and $\mathbb{E}^{\alpha, n}$ the corresponding probability measure and expectation.
	
	For the game with an unbounded number of shocks we define the process $Z$ analogously to \eqref{eq:model_Z_dynamics} via $Z_0=0$ and
	\begin{align} \label{eq:InfShocksZ}
		\de Z_t &= \sum_{k=1}^\infty \mathbb{I}_{\{Z_{t-}=k-1\}} \de N_t^k.
	\end{align} 
	As before we define the filtration generated by common shocks events $\mathfrak{G} = (\mathfrak{G}_t)_{t \in [0,T]}$ via \eqref{eq:def_filtration_g}. Given a $\mathfrak{G}$-adapted process $M$, the state process $X$ is given by
	\begin{align*}
		\de X_t &= \sum_{i,j \in \mathbb{S}: i \neq j} \sum_{k=0}^\infty \mathbb{I}_{\{X_{t-}=i\}} \mathbb{I}_{\{Z_{t-}=k\}} (j-i) \de N_t^{ikj} \\ 
		&\quad + \sum_{i\in \mathbb{S}} \sum_{k=1}^\infty \mathbb{I}_{\{X_{t-}=i\}} \mathbb{I}_{\{Z_{t-}=k-1\}} \left( J^i(t,M_{t-})-i \right) \de N^k_t.
	\end{align*}
	We highlight that the only difference between $X$ and $X^n$ is that in the definition the sum run to infinity rather than to $n$. 
	In particular, it is clear that $X=X^n$ and $Z=Z^n$ on $\{Z_T \le n\}$.
	
	Strategies are defined analogously as in the $n$-shock case, with $\mathcal{A}$ denoting the set of all admissible strategies for the mean field game with an unbounded number of shocks. 
	If $\alpha$ is a strategy such that
	\[
	\alpha(t,(Z_{(\cdot \wedge t)-}, X_{(\cdot \wedge t)-}) = \nu(t, (Z_{(\cdot \wedge t)-}, X_{t-})
	\]
	for some function $\nu: [0,T] \times \mathbb{N}^{[0,T]} \times \mathbb{S} \rightarrow \mathbb{A}$, we refer to $\alpha$ (and $\nu$) as a feedback strategy.
	Finally, as in the $n$-shock case the model is specified by the functions $Q$, $\lambda$, $\psi$ and $\Psi$, which have exactly the same intuitive interpretations; we assume that
	\begin{align*}
		\lambda &: \mathbb{N} \times [0,T] \times \mathcal{P}(\mathbb{S}) \rightarrow (0,\infty) \\
		Q &: [0,T] \times \mathbb{N}_0 \times \mathcal{P}(\mathbb{S}) \times \mathbb{A} \rightarrow \mathbb{R}^{S \times S} \\
		\psi &: [0,T] \times \mathbb{N}_0\times \mathcal{P}(\mathbb{S}) \times \mathbb{A} \rightarrow \mathbb{R}^S \\
		\Psi &: \mathbb{N}_0 \times \mathcal{P}(\mathbb{S}) \rightarrow \mathbb{R}^S\\
		J &:[0,T] \times \mathcal{P}(\mathbb{S}) \rightarrow \mathbb{S}^\mathbb{S}
	\end{align*} are bounded and Borel measurable and that for all $t \in [0,T]$, $k \in \mathbb{N}_0$, $m \in \mathcal{P}(\mathbb{S})$ and $a \in \mathbb{A}$ the matrix $Q(t,k,m,a)$ is an intensity matrix.
	
	As before given the common shocks process $Z$ and a fixed exogenous aggregate distribution process $M$, the individual agent aims to maximize
	\[
	\mathbb{E}^{\alpha, \infty} \left[ \int_0^T \psi^{X_t} (t, Z_t, M_t, \alpha_t) \de t + \Psi^{X_T}(Z_T,M_T) \right]\quad\text{over all }\alpha\in\mathcal{A},
	\] where $\mathbb{E}^{\alpha, \infty}$ denotes expectation with respect to the equivalent probability measure $\mathbb{P}^{\alpha, \infty}$.
	This measure is again defined in such a way that the the common shocks process $Z$ jumps from $k-1$ to $k$ with intensity $\lambda^k$, and the agent's state process $X$ jumps from $i$ to $j$ with intensity $Q^{ikj}$ when $k$ common shocks have occurred.
	The construction of this measure change is mathematically slightly more delicate, since here we need to change the intensities of an infinite number of Poisson process simultaneously.
	For details on the construction and the explicit definition of $\mathbb{P}^{\alpha, \infty}$ we refer to Appendix~\ref{sec:measure_change_unbounded}. 
	
	The definition of mean field equilibria is exactly the same as Definition~\ref{def:MFE}.
	In the present context, we focus on approximate $\epsilon$-equilibria $(\hat{\alpha},\hat{M})$, where given the process $\hat{M}$ an agent that deviates from playing $\hat{\alpha}$ can improve by at most $\epsilon$ compared to the reward earned by implementing $\hat{\alpha}$.
	Formally, approximate equilibria are defined as follows:
	
	\begin{definition}
		An \emph{$\epsilon$-equilibrium} is a pair $(\hat{\alpha}, \hat{M})$ consisting of an admissible strategy $\hat{\alpha} \in \mathcal{A}$ and a $\mathcal{P}(\mathbb{S})$-valued, $\mathfrak{G}$-adapted process $\hat{M}$ such that
		\begin{align*}
			&\mathbb{E}^{\alpha, \infty} \left[ \int_0^T \psi^{X_t} (t, Z_t, \hat M_t, \alpha_t) \de t + \Psi^{X_T}(Z_T,\hat M_T) \right] \\ 
			&\le \mathbb{E}^{\hat\alpha, \infty} \left[ \int_0^T \psi^{X_t} (t, Z_t,\hat M_t, \hat\alpha_t) \de t + \Psi^{X_T}(Z_T,\hat M_T) \right] +\epsilon
		\end{align*} 
		for every strategy $\alpha \in \mathcal{A}$, while at the same time
		\[
		\mathbb{P}^{\hat\alpha, \infty}(X_t =\cdot \;|\mathfrak{G}_t) = \hat M_t\quad\text{for all }t\in[0,T]. 
		\] 
	\end{definition}

	In the remainder of this section we prove that equilibria in feedback strategies for the game with $n$ shocks are $\epsilon_n$-equilibria for the game with an unbounded number of shocks, where the sequence $\epsilon_n$, $n\in\mathbb{N}$ decays exponentially fast. 
	To state this result, we describe how we formally relate strategies for the game with at most $n$ shocks to strategies for the game with an unbounded number of shocks:
	Namely, for a strategy $\alpha^n$ of the game with at most $n$ shocks we define the associated strategy
	\[
	\tilde{\alpha}^n(t, X_{(\cdot \wedge t)-}, Z_{(\cdot \wedge t)-}) = \alpha^n(t, X_{(\cdot \wedge t)-}, {Z_{(\cdot \wedge t)-} \wedge n})
	\]
	for the game with an unbounded number of shocks.
	This strategy chooses the same action as $\alpha^n$ at time $t$ whenever $Z_t \le n$; and if $Z_n >n$, it ignores the information that more than $n$ shocks occurred and continues to implement $\alpha^n(t,X_{(\cdot \wedge t)-}, n)$.
	
	With these preparations we can now construct an $\epsilon_n$-equilibrium from a mean field equilibrium in feedback strategies of the game with up to $n$ shocks.
	
	\begin{theorem}
		\label{thm:Approximation}
		Let $(\alpha^n,M^n)$ be an equilibrium for the mean field game with up to $n$ shocks, where $\alpha^n$ is a feedback strategy. 
		Then $(\tilde{\alpha}^n,\tilde M^n)$ with $\tilde M^n=M^{\tilde{\alpha}^n}$ as in Lemma~\ref{Lemma:PopDistGivenAlpha} is an $\epsilon_n$-equilibrium for the game with infinitely many shocks. The constant $\epsilon_n$ is given by
		\[
		\epsilon_n = 4 (\psi_\text{max}T + \Psi_\text{max}) \left(1- \exp(-\max \{\lambda_\text{max},1\}T) \right)^n
		\] where the constants are as in \eqref{eq:appendix_constants}.
		Moreover, $|\tilde{M}^n - M^n| \overset{\mathbb{P}}{\rightarrow} 0$ for $n \rightarrow \infty$.
	\end{theorem}
	
	To prove the result let us introduce the following notations. 
	Let $M$ be a $\mathfrak{G}$-adapted process, $M^n$ be a $\mathfrak{G}^n$-adapted process, $\alpha$ be an admissible strategy for the game with an unbounded number of shocks and $\alpha^n$ be an admissible strategy for the game with at most $n$ shocks. 
	Then we define the value of a strategy $\alpha$ given the process $M$ in the game with an unbounded number of shocks as
	\begin{align*}
		V^\infty(\alpha, M) &= \mathbb{E}^{\alpha, \infty} \left[ \int_0^T \psi^{X_t}(t,Z_t, M_t, \alpha_t) \de t + \Psi^{X_T}(Z_T,M_T) \right]    
	\end{align*} 
	and similarly, the value of the strategy $\alpha^n$ given the process $M^n$ in the game with at most $n$ shocks as
	\[
	V^n(\alpha^n, M^n) = \mathbb{E}^{\alpha^n,n} \left[ \int_0^T \psi^{X_t^n} (t, Z_t^n, M_t^n, \alpha_t^n) \de t + \Psi^{X_T^n}(Z_T^n, M_T^n) \right].
	\]
	
	Before we provide the proof of Theorem~\ref{thm:Approximation}, we link the value functions of the individual agent's control problems and the aggregate distributions with finitely many and infinitely many shocks, respectively.
	
	\begin{lemma}
		\label{lemma:ApproxValueFkt}
		Let $n \in \mathbb{N}$. Moreover, assume that $M$ is a $\mathfrak{G}$-adapted process, $M^n$ is a $\mathfrak{G}^n$-adapted process, $\alpha \in \mathcal{A}$ and $\alpha^n \in \mathcal{A}^n$ are strategies such that $M_t=M_t^n$ and $\alpha = \alpha^n$ on $\{Z_t \le n\}$. Then
		\begin{align*}
			\left| V^\infty(\alpha, M) - V^n(\alpha^n, M^n) \right| \le  2(\psi_\text{max}T + \Psi_\text{max}) \left(1- \exp(-\max \{\lambda_\text{max},1\} T) \right)^n.
		\end{align*}
	\end{lemma}
	
	\begin{proof} 
		We have
		\begin{align*}
			&\left| V^\infty(\alpha, M) - V^n(\alpha^n, M^n) \right| \\
			&= \left| \mathbb{E}^{\alpha, \infty} \left[ \int_0^T \psi^{X_t} (t, Z_t, M_t, \alpha_t) \de t + \Psi^{X_T}(Z_T,M_T) \right] \right. \\
			&\quad \left. - \mathbb{E}^{\alpha^n,n} \left[ \int_0^T \psi^{X_t^n} (t, Z_t^n, M_t^n, \alpha_t^n) \de t + \Psi^{X_T^n}(Z_T^n,M_T^n) \right] \right|\\
			&= \left| \mathbb{E} \left[ \Theta_T^{\alpha, \infty} \left( \mathbb{I}_{\{Z_T \le n\}} \left( \int_0^T \psi^{X_t}(t,Z_t, M_t, \alpha_t) \de t + \Psi^{X_T}(Z_T, M_t) \right) \right. \right. \right. \\
			& \qquad \left. \left. + \mathbb{I}_{\{Z_T >n\}} \left( \int_0^T \psi^{X_t}(t,Z_t, M_t, \alpha_t) \de t + \Psi^{X_T}(Z_T,M_T) \right) \right) \right]. \\
			& \quad  - \mathbb{E} \left[ \Theta_T^{\alpha^n,n} \left( \mathbb{I}_{\{Z_T \le n\}} \left( \int_0^T \psi^{X_t^n}(t,Z_t^n, M_t^n, \alpha^n_t) \de t + \Psi^{X_T^n}(Z_T^n, M_T^n) \right) \right. \right. \\
			& \qquad \left. \left. \left. + \mathbb{I}_{\{Z_T >n\}} \left( \int_0^T \psi^{X_t^n}(t,Z_t^n, M_t^n, \alpha_t^n) \de t + \Psi^{X_T^n}(Z_T^n,M_T^n) \right) \right) \right] \right| .
		\end{align*} Since on $\{Z_T \le n\}$ we have $X=X^n$, $Z=Z^n$, $\alpha=\alpha^n$, $\Theta_T^{\alpha, \infty} = \Theta_T^{\alpha^n,n}$ and $M_t = M_t^n$, we obtain
		\begin{align*}
			&\left| V^\infty(\alpha, M) - V^n(\alpha^n, M^n) \right| \\
			&= \left| \mathbb{E}^{\alpha, \infty} \left[   \mathbb{I}_{\{Z_T >n\}} \left( \int_0^T \psi^{X_t}(t,Z_t, M_t, \alpha_t) \de t + \Psi^{X_T}(Z_T,M_T) \right) \right] \right. \\
			& \quad - \left. \mathbb{E}^{\alpha^n,n} \left[  \mathbb{I}_{\{Z_T >n\}} \left( \int_0^T \psi^{X_t^n}(t,Z_t^n, M_t^n, \alpha_t^n) \de t + \Psi^{X_T^n}(Z_T^n,M_T^n) \right)  \right] \right| \\
			&\le \mathbb{P}^{\alpha, \infty}(Z_T>n) (\psi_\text{max}T + \Psi_\text{max}) + \mathbb{P}^{\alpha^n,n}(Z_T>n) (\psi_\text{max}T + \Psi_\text{max}).
		\end{align*} Moreover, we have that
		\begin{align*}
			\mathbb{P}^{\alpha, \infty}(Z_T >n) &\le \mathbb{P}^{\alpha,\infty} \left(N_T^k \ge 1 \quad \text{for all } k \in \{1, \ldots, n\}\right) \\
			&= \prod_{k=1}^n \left( 1- \exp \left( - \int_0^T \left( \lambda^k(t,M_t) \mathbb{I}_{\{Z_t = k-1\}} + \mathbb{I}_{\{Z_t \neq k-1\}} \right) \de t \right) \right) \\
			&\le \prod_{k=1}^n \left( 1 - \exp \left( - \max \{\lambda_\text{max}, 1\} T \right) \right) \\
			&= \left( 1 - \exp \left( - \max \{\lambda_\text{max}, 1\} T \right) \right)^n.
		\end{align*} Analogously, we obtain for $\alpha^n$ that
		\begin{align*}
			\mathbb{P}^{\alpha^n,n}(\tilde{Z}_T >n) &\le \left( 1 - \exp \left( - \max \{\lambda_\text{max}, 1\} T \right) \right)^n,
		\end{align*} which completes the proof.
	\end{proof}
	
	\begin{lemma}[aggregation]
		\label{Lemma:PopDistGivenAlpha}
		Let $\alpha \in \mathcal{A}$ be a feedback strategy.
		Then, using $\tau_0=0$ and $\tau_k = \inf \{s \ge \tau_{k-1} : Z_s = k\}$, $k \in \mathbb{N}$, the process $M^\alpha$ given by 
		\[
		(M_t^\alpha)^j (\omega) = \begin{cases}
			(M_{\tau_{k-1}}^\alpha)^j + \int_{\tau_k}^t \sum_{i \in \mathbb{S}} (M_s^\alpha)^i Q^{ij}(s, Z_{s-}, M_s^\alpha, \nu(s,Z_{(\cdot \wedge s)-}, i)) \de s\quad&\text{if } \tau_{k-1}(\omega) < t < \tau_k (\omega) \\
			\sum_{i \in \mathbb{S}} \mathbb{I}_{\{J^i(t,M_{t-})=j\}} (M_{t-}^\alpha)^i &\text{if } t=\tau_k(\omega), k \in \mathbb{N} \\
			\mathbb{P}(X_0=j) &\text{if } t=0
		\end{cases}
		\] 
		is $\mathfrak{G}$-adapted and $\mathcal{P}(\mathbb{S})$-valued.
		Moreover, it satisfies $M_t^\alpha = \mathbb{P}^{\alpha,\infty}(X_t =\cdot|\mathfrak{G}_t)$ for all $t \in [0,T]$.
	\end{lemma}
	
	\begin{proof}
		The process $M^\alpha$ is well-defined as the unique solution of an ordinary differential equation with a Lipschitz continuous right-hand side on each interval $[\tau_{k-1}(\omega),\tau_k(\omega))$, $k \in \mathbb{N}$.
		Moreover, it is immediate that $M_t^\alpha$ is $\mathfrak{G}_t$-measurable and that $M_t^\alpha \in \mathcal{P}(\mathbb{S})$.
		Furthermore, as in Lemma~\ref{Lemma:DistGivenMarkov}, we obtain that 
		\begin{align*}
			&\mathbb{P}^{\alpha, \infty}(X_t = l|\mathfrak{G}_T) \\
			&= \mathbb{P}^{\alpha, \infty}(X_0=l) + \sum_{i \in \mathbb{S}} \int_0^t \mathbb{P}^{\alpha, \infty} (X_{s-} = i|\mathfrak{G}_T) Q^{il}(s,Z_{s-}, M_s^\alpha, \alpha(s,Z_{(\cdot \wedge s)-},i)) \de s \\
			&\quad + \sum_{k=1}^{Z_t} \int_0^t \mathbb{I}_{\{Z_{s-} = k-1\}} \left( \sum_{i \in \mathbb{S}} \mathbb{I}_{\{J^i(s,M_{s-})=l\}} \mathbb{P}^{\alpha, \infty}(X_{s-}=i|\mathfrak{G}_T) - \mathbb{P}^{\alpha, \infty} (X_{s-}=l|\mathfrak{G}_T) \right) \de N_s^k.
		\end{align*} This implies the desired result.
	\end{proof}
	
	Using Lemmas~\ref{lemma:ApproxValueFkt} and \ref{Lemma:PopDistGivenAlpha} it is now easy to complete the proof of Theorem~\ref{thm:Approximation}:
	
	\begin{proof}[Proof of Theorem~\ref{thm:Approximation}]
		We first note that $\alpha^n = \tilde{\alpha}^n$ on $\{Z_T \le n\}$. 
		The processes $M^n$ and $\tilde{M}^n = M^{\tilde{\alpha}^n}$ are defined pathwise and the definitions agree on $\{Z_T \le n\}$ (see Lemma~\ref{Lemma:PopDistGivenAlpha}).
		Hence, $\tilde{M}^n$ and $M^n$ coincide on $\{Z_T \le n\}$. 
		Noting that $\mathbb{P}(Z_T >n)\le(1-\exp(-T))^n \rightarrow 0$, we have the desired convergence of $|\tilde{M}^n - M^n|$ in probability. Moreover, by Lemma~\ref{Lemma:PopDistGivenAlpha} the process $\tilde{M}^n$ satisfies $\mathbb{P}^{\tilde{\alpha}^n, \infty}(X_t = \cdot \; |\mathfrak{G}_t) = \tilde{M}_t^n$ for all $t \in [0,T]$.
		
		Now, let $\alpha$ be an arbitrary admissible strategy for the game with infinitely many shocks.
		Then by Lemma~\ref{lemma:ApproxValueFkt}
		\begin{align*}
			V^\infty(\alpha, \tilde{M}^n) \le V^n(\alpha, M^n) + \tfrac{1}{2} \epsilon_n.
		\end{align*} Since $\alpha^n$ is an equilibrium for the game with $n$ shocks, we obtain
		\begin{align*}
			V^n(\alpha, M^n) + \tfrac{1}{2} \epsilon_n \le V^n(\alpha^n, M^n) + \tfrac{1}{2} \epsilon_n.
		\end{align*} 
		Finally, another application of Lemma~\ref{lemma:ApproxValueFkt} yields 
		\begin{align*}
			V^n(\alpha^n, M^n) + \tfrac{1}{2} \epsilon_n \le V^\infty(\tilde{\alpha}^n, \tilde{M}^n) +  \epsilon_n.
		\end{align*} 
		Combining all three inequalities yields the desired claim.
	\end{proof}

	\section{Application: Corruption Detection with Random Audits} \label{sec:corruption}
	
	We consider a variant of the corruption detection model in \cite{KolokoltsovCorruption2017} that includes unpredictable audits at random times, where corrupt agents are identified and penalized.
	As in the original model, the agents' state space is $\mathbb{S} = \{C,H,R\}$, where $C$ means that the agent is corrupt, $H$ means that the agent is honest and $R$ means that the agent is in a penalized state. 
	The agent's income satisfies $r_C\ge r_H\ge r_R \ge 0$, and in states $C$ and $H$ the agent has the possibility to exert effort to switch to another state. Moreover, he faces peer pressure in two ways: First, the higher the percentage of corrupt agents, the larger becomes the rate for an honest agent to also turn corrupt. Second, the smaller the fraction of corrupt agents, the larger is the detection rate and, thus, the risk to be caught and transferred to the penalized state $R$. Once an agent is caught and moved to state $R$, he returns to the honest state with a fixed rate $q_\text{rec}$.
	The novel aspect in our model is the presence of random audits: These occur at a rate $\lambda$ and result in all corrupt agents being transferred to state $R$ immediately.
	Figure~\ref{figure:Corruption} displays the model dynamics, including state transitions due to the agent's individual initiative and due to audits where $q_\text{soc}, q_\text{inf}, q_\text{rec} \ge 0$.
	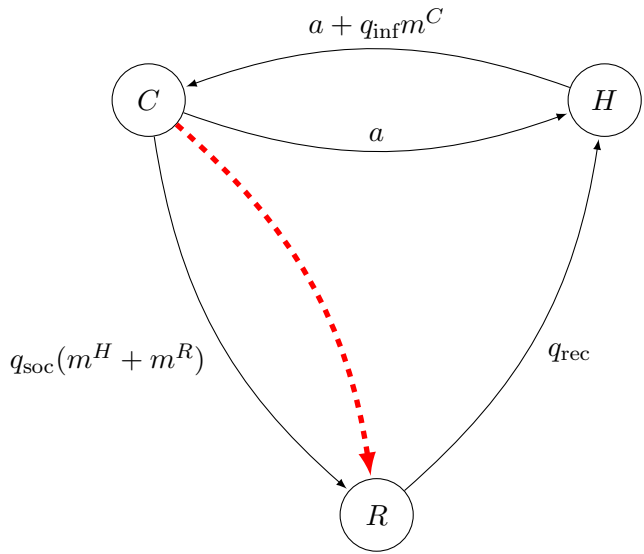
\begin{figure}
		\begin{center}
			\begin{tikzpicture}
				
				\node[state] at (0,0) (A)     {$C$};
				\node[state] at (6,0) (B)     {$H$};
				\node[state] at (3, -5.5)  (C)     {$R$};
				
				\draw[every loop, >=latex, auto=right] 
				(A) edge[bend right=20, auto=left] node {$a$} (B)
				(B) edge[bend right=20] node {$a+q_\text{inf}m^C$} (A)
				(A) edge[bend right=20] node {$q_\text{soc}(m^H+m^R)$} (C)
				(C) edge[bend right=20] node {$q_\text{rec}$} (B);
				
				\draw[every loop, >=latex, red, dashed, line width = 0.7mm] 
				(A) edge[bend left=20] (C);
				
			\end{tikzpicture}
		\end{center}
		\caption{Dynamics in the corruption detection model. The red dashed line represents the transitions due to audits.}
		\label{figure:Corruption}
	\end{figure}
	The agent aims to maximize
	\[
	\mathbb{E}^\alpha \left[ \int_0^T \left(r^{X_t} - \tfrac{1}{2} \alpha_t^2 \right)\de t \right].
	\]
	
	Formally, the relevant model coefficients are thus given by
	\begin{align*}
		Q(t,k,m,a) &= \begin{pmatrix}
			- q_\text{soc} (m^H+m^R) - a & a & q_\text{soc} (m^H+m^R) \\
			a + q_\text{inf} m^C & - a-q_\text{inf}m^C & 0 \\
			0 & q_\text{rec} & - q_\text{rec}
		\end{pmatrix} \\
		\psi(t,k,m,a) &= \begin{pmatrix}
			r^C - \frac{1}{2} a^2 \\
			r^H - \frac{1}{2} a^2 \\
			r^R
		\end{pmatrix} \\
		J(t,k,m) &= \begin{pmatrix}
			R \\ H \\ R
		\end{pmatrix}.
	\end{align*}
	For our numerical results, we set the initial distribution to $m = (0.2, 0.8,0)$, and the relevant parameters as reported in Table~\ref{table:Corruption}.
	
	\begin{table}[h!]
		\begin{center}
			\begin{tabular}{cccccccccccccc}
				\hline
				Parameter & $T$ & $n$ & $\lambda$ & $q_\text{inf}$ & $q_\text{soc}$ & $q_\text{rec}$ & $c_\text{dec}$ & $r^C$ & $r^H$ & $r^R$ \\
				\hline
				Value & 2 & 2 & 2 & 5 & 2 & 0.5 & 2 & 10 & 5 & 0 \\
				\hline
			\end{tabular}
		\end{center}
		\caption{Coefficients for the corruption detection model.}
		\label{table:Corruption}
	\end{table}

	\begin{figure}
		\centering
		\begin{subfigure}[b]{0.49\textwidth}
			\includegraphics[width=\textwidth]{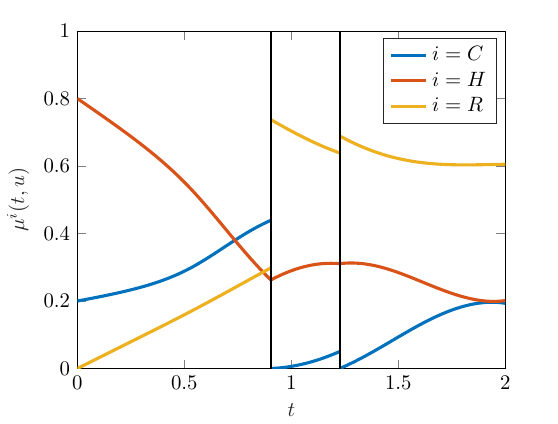}
		\end{subfigure}
		\hfill
		\begin{subfigure}[b]{0.49\textwidth}
			\includegraphics[width=\textwidth]{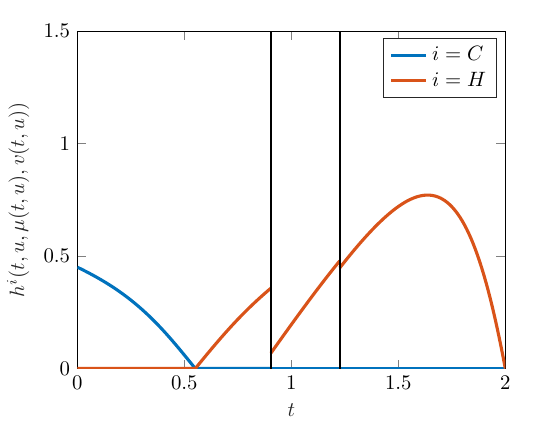}
		\end{subfigure}
		\caption{Evolution of the equilibrium distribution (left) and the equilibrium control (right) for one sample path. Shock times (audits) are represented by black vertical lines.}
		\label{fig:Corruption_Control}
	\end{figure}
	
	\begin{figure}
		\centering
		\begin{subfigure}[b]{0.49\textwidth}
			\includegraphics[width=\textwidth]{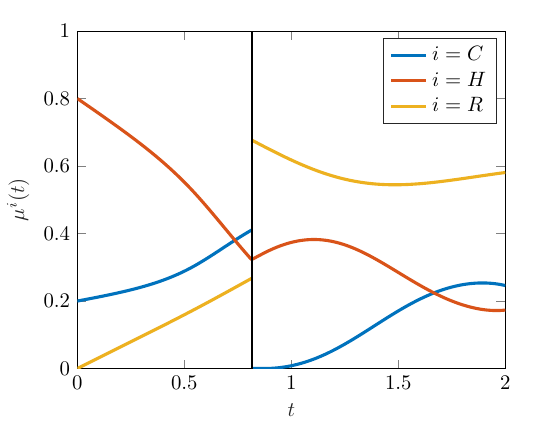}
		\end{subfigure}
		\hfill
		\begin{subfigure}[b]{0.49\textwidth}
			\includegraphics[width=\textwidth]{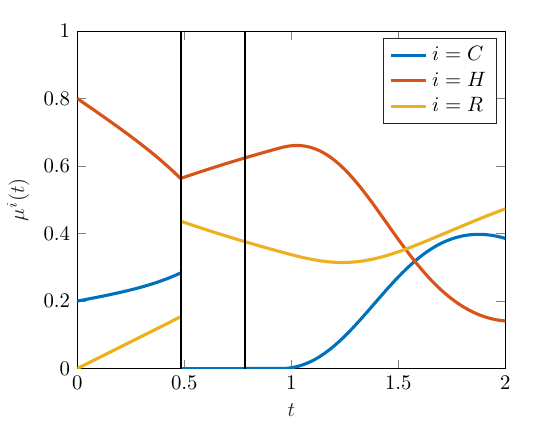}
		\end{subfigure}
		\newline
		\centering
		\begin{subfigure}[b]{0.49\textwidth}
			\includegraphics[width=\textwidth]{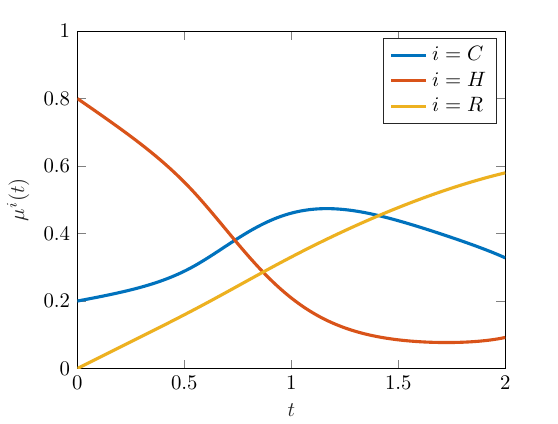}
		\end{subfigure}
		\hfill
		\begin{subfigure}[b]{0.49\textwidth}
			\includegraphics[width=\textwidth]{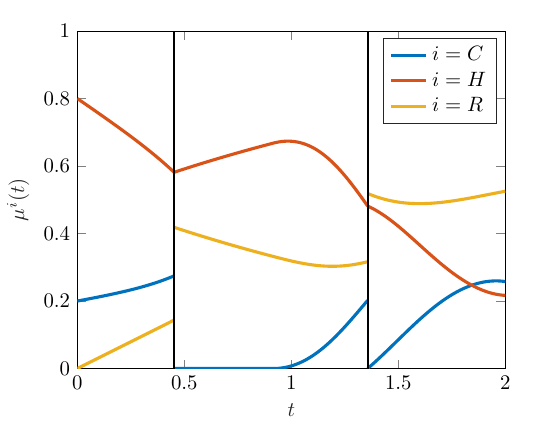}
		\end{subfigure}
		\caption{Evolution of the equilibrium distribution for different sample paths. Shock times (audits) are represented by black vertical lines.}
		\label{fig:CorruptionFirst}
	\end{figure}
	
	\begin{figure}
		\centering
		\includegraphics[width=0.5\textwidth]{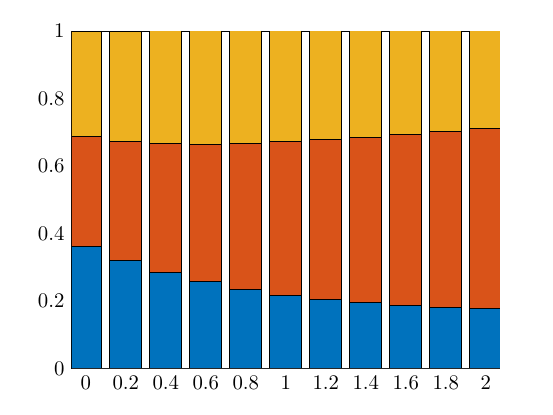}
		\caption{Expected overall shares of agents: corrupt (blue, bottom), honest (red, middle), reserved (yellow, top), for varying levels of the audit intensity $\lambda$. The first column ($\lambda=0$) represents the special case of the model without audits.}
		\label{fig:corruption_effect}
	\end{figure}
	
	\begin{figure}
		\centering
		\includegraphics[width=0.5\textwidth]{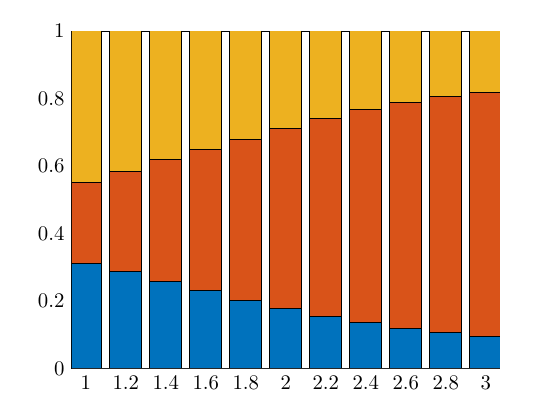}
		\caption{Expected overall shares of agents: corrupt (blue, bottom), honest (red, middle), reserved (yellow, top), for varying levels of the peer pressure intensity $q_\text{soc}$.}
		\label{fig:corruption_peer_pressure_effect}
	\end{figure}
	
	Figure~\ref{fig:Corruption_Control} shows one sample path of the equilibrium distribution and control.
	Figure~\ref{fig:CorruptionFirst} illustrates four sample paths of the equilibrium distribution for different realizations of the driving randomness.
	We observe that the equilibrium distribution is mainly governed by two effects: When the environment is corrupted in the sense that the share of corrupt agents is large, the likelihood for honest agents to turn corrupt increases. Conversely, when there are no or few corrupt agents, the probability of being identified as corrupt increases for corrupt agents, making it relatively more attractive to become honest.
	In addition, all corrupt agents are moved to the penalized state upon audits (vertical black lines).
	
	The average reduction of the share of corrupt agents due to the presence of random audits can be quantified.
	To wit, we compute the expected share of agents averaged over time, via $\tfrac{1}{T}\mathbb{E} \left[\int_0^T \mu^i(t,U_t) \de t\right]$ for all three states $i \in \{C,H,R\}$.
	We compare the resulting shares without audits (first bar) and for different values of the intensity of audits ranging from $0.2$ to $2$ in Figure~\ref{fig:corruption_effect}. The share of corrupt agents is decreasing with $\lambda$, whereas the share of honest agents is increasing.
	The main transmission channel of audits is through the creation of a corruption-free environment immediately after each audit, making the environment more honest over time on aggregate and thus -- via the "peer pressure" effect -- reducing the payoff of being corrupt. The relevance of this "peer pressure" effect is illustrated in Figure~\ref{fig:corruption_peer_pressure_effect}. Here, we again observe that an increase in the peer pressure through the rate $q_\text{soc}$ leads to a smaller share of corrupt agents and a larger share of honest agents.
	
	\bibliography{literatureCommonShocks} 		

\begin{thebibliography}{10}

\bibitem{AhujaCommonNoise}
Saran Ahuja.
\newblock Wellposedness of mean field games with common noise under a weak
  monotonicity condition.
\newblock {\em SIAM Journal on Control and Optimization}, 54(1):30--48, 2016.

\bibitem{AurellEpidemics}
Alexander Aurell, Ren\'{e} Carmona, G\"{o}k\c{c}e Dayanikli, and Mathieu
  Lauri\`{e}re.
\newblock {Optimal Incentives to Mitigate Epidemics: A Stackelberg Mean Field
  Game Approach}.
\newblock {\em SIAM Journal on Control and Optimization}, 60(2):S294--S322,
  2022.

\bibitem{BayraktarFiniteStateCommonNoise}
Erhan Bayraktar, Alekos Cecchin, Asaf Cohen, and Fran\c{c}ois Delarue.
\newblock {Finite state mean field games with Wright-Fisher common noise}.
\newblock {\em Journal de Math\'{e}matiques Pures et Appliqu\'{e}es},
  147:98--162, 2021.

\bibitem{BayraktarFiniteStateCommonNoise_2}
Erhan Bayraktar, Alekos Cecchin, Asaf Cohen, and Fran\c{c}ois Delarue.
\newblock {Finite State Mean Field Games with Wright-Fisher Common Noise as
  Limits of N-Player Weighted Games}.
\newblock {\em Mathematics of Operations Research}, 47(4):2840--2890, 2022.

\bibitem{BayraktarMasterEquation}
Erhan Bayraktar and Asaf Cohen.
\newblock {Analysis of a Finite State Many Player Game Using Its Master
  Equation}.
\newblock {\em SIAM Journal on Control and Optimization}, 56(5):3538--3568,
  2018.

\bibitem{BelakMFG}
Christoph Belak, Daniel Hoffmann, and Frank~T. Seifried.
\newblock {Continuous-Time Mean Field Games with Finite State Space and Common
  Noise}.
\newblock {\em Applied Mathematics \& Optimization}, 84:3173--3216, 2021.

\bibitem{BensoussanMFG}
Alain Bensoussan, Jens Frehse, and Phillip Yam.
\newblock {\em Mean Field Games and Mean Field Type Control Theory}.
\newblock SpringerBriefs in Mathematics. Springer, 2013.

\bibitem{BertucciRemarks}
Charles Bertucci, Jean-Michel Lasry, and Pierre-Louis Lions.
\newblock Some remarks on mean field games.
\newblock {\em Communications in Partial Differential Equations},
  44(3):205--227, 2019.

\bibitem{CardaliaguetLectureNotes}
Pierre Cardaliaguet.
\newblock {Notes on Mean Field Games (from P.-L. Lions' Lectures at Coll\`{e}ge
  de France)}.
\newblock \url{https://www.ceremade.dauphine.fr/~cardaliaguet/MFG20130420.pdf},
  2013.

\bibitem{Cardaliaguet2022First_Order_Common}
Pierre Cardaliaguet and Panagiotis~E. Souganidis.
\newblock {On first order mean field game systems with a common noise}.
\newblock {\em The Annals of Applied Probability}, 32(3):2289--2326, 2022.

\bibitem{CarjaFlowInvCaratheodory}
Ovidiu C\^{a}rj\v{a} and Ioan~I. Vrabie.
\newblock {Differential Equations on Closed Sets}.
\newblock In A.~Cañada, P.~Drábek, and A.~Fonda, editors, {\em Handbook of
  Differential Equations: Ordinary Differential Equations}, volume~2,
  chapter~3, pages 147--238. Elsevier, 2006.

\bibitem{CarmonaMFG2018}
Ren\'{e} Carmona and Fran\c{c}ois Delarue.
\newblock {\em Probabilistic Theory of Mean Field Games with Applications I:
  Mean Field FBSDEs, Control, and Games}, volume~83 of {\em Probability Theory
  and Stochastic Modelling}.
\newblock Springer International Publishing, 2018.

\bibitem{CarmonaMFGPartTwo2018}
Ren\'{e} Carmona and Fran\c{c}ois Delarue.
\newblock {\em Probabilistic Theory of Mean Field Games with Applications II:
  Mean Field Games with Common Noise and Master Equations}, volume~84 of {\em
  Probability Theory and Stochastic Modelling}.
\newblock Springer International Publishing, 2018.

\bibitem{CarmonaCommonNoise}
Ren\'{e} Carmona, Fran\c{c}ois Delarue, and Daniel Lacker.
\newblock Mean field games with common noise.
\newblock {\em The Annals of Probability}, 44(6):3740--3803, 2016.

\bibitem{CarmonaSystemicRisk}
Ren\'{e} Carmona, Jean-Pierre Fouque, and Li-Hsien Sun.
\newblock Mean field games and systemic risk.
\newblock {\em Communications in Mathematical Sciences}, 13(4):911--933, 2015.

\bibitem{CarmonaExtended}
Ren\'{e} Carmona and Peiqi Wang.
\newblock {A Probabilistic Approach to Extended Finite State Mean Field Games}.
\newblock {\em Mathematics of Operations Research}, 46(2):471--502, 2021.

\bibitem{Cecchin_Vanishing_Common_Noise}
Alekos Cecchin and Fran\c{c}ois Delarue.
\newblock Selection by vanishing common noise for potential finite state mean
  field games.
\newblock {\em Communications in Partial Differential Equations},
  47(1):89--168, 2022.

\bibitem{CecchinProbabilistic2018}
Alekos Cecchin and Markus Fischer.
\newblock {Probabilistic Approach to Finite State Mean Field Games}.
\newblock {\em Applied Mathematics \& Optimization}, 81:253--300, 2020.

\bibitem{CecchinMasterEquation}
Alekos Cecchin and Guglielmo Pelino.
\newblock {Convergence, fluctuations and large deviations for finite state mean
  field games via the Master Equation}.
\newblock {\em Stochastic Processes and their Applications},
  129(11):4510--4555, 2019.

\bibitem{DelarueFiniteCommonNoise}
Fran\c{c}ois Delarue.
\newblock {Master Equation for Finite State Mean Field Games with Additive
  Common Noise}.
\newblock In Pierre Cardaliaguet and Alessio Porretta, editors, {\em Mean Field
  Games}, volume 2281 of {\em Lecture Notes in Mathematics}, pages 203--248.
  Springer, 2020.

\bibitem{DoncelPaper}
Josu Doncel, Nicolas Gast, and Bruno Gaujal.
\newblock {Discrete mean field games: Existence of equilibria and convergence}.
\newblock {\em Journal of Dynamics and Games}, 6(3):221--239, 2019.

\bibitem{doncel_gast_gaujal_2022_sir}
Josu Doncel, Nicolas Gast, and Bruno Gaujal.
\newblock {A mean field game analysis of SIR dynamics with vaccination}.
\newblock {\em Probability in the Engineering and Informational Sciences},
  36(2):482–499, 2022.

\bibitem{dumitrescu2023energy}
Roxana Dumitrescu, Marcos Leutscher, and Peter Tankov.
\newblock Energy transition under scenario uncertainty: a mean-field game of
  stopping with common noise.
\newblock ArXiv:2210.03554, 2023.

\bibitem{escribe_renewable_investment}
C{\'e}lia Escribe, Josselin Garnier, and Emmanuel Gobet.
\newblock {A Mean Field Game Model for Renewable Investment under Long-Term
  Uncertainty and Risk Aversion}.
\newblock \url{https://hal.science/hal-04055421}, 2023.

\bibitem{FilippovDiscontiODE1988}
A.~F. Filippov.
\newblock {\em Differential Equations with Discontinuous Righthand Sides},
  volume~18 of {\em Mathematics and its Applications}.
\newblock Springer Netherlands, 1988.

\bibitem{GomesSocio2014}
Diogo Gomes, Roberto~M. Velho, and Marie-Therese Wolfram.
\newblock Socio-economic applications of finite state mean field games.
\newblock {\em Philosophical Transactions of the Royal Society of London A:
  Mathematical, Physical and Engineering Sciences}, 372(2028), 2014.

\bibitem{GomesConti2013}
Diogo~A. Gomes, Joana Mohr, and Rafael~Rig{\~{a}}o Souza.
\newblock {Continuous Time Finite State Mean Field Games}.
\newblock {\em Applied Mathematics \& Optimization}, 68(1):99--143, 2013.

\bibitem{GueantCongestion2015}
Olivier Gu\'{e}ant.
\newblock {Existence and Uniqueness Result for Mean Field Games with Congestion
  Effect on Graphs}.
\newblock {\em Applied Mathematics \& Optimization}, 72(2):291--303, 2015.

\bibitem{ParisPrinceton2010}
Olivier Gu\'{e}ant, Jean-Michel Lasry, and Pierre-Louis Lions.
\newblock {Mean Field Games and Applications}.
\newblock In {\em Paris-Princeton Lectures on Mathematical Finance 2010},
  volume 2003 of {\em Lecture Notes in Mathematics}, pages 205--266.
  Springer-Verlag, 2011.

\bibitem{HuangNCE2006}
Minyi Huang, Roland~P. Malham\'{e}, and Peter~E. Caines.
\newblock {Large population stochastic dynamic games: closed-loop McKean-Vlasov
  systems and the Nash certainty equivalence principle}.
\newblock {\em Communications in Information \& Systems}, 6(3):221--252, 2006.

\bibitem{Kallenberg2002Foundations}
Olav Kallenberg.
\newblock {\em Foundations of Modern Probability}.
\newblock Probability and Its Applications. Springer, 2\textsuperscript{nd}
  edition, 2002.

\bibitem{KolokoltsovBotnet2016}
V.~N. Kolokoltsov and A.~Bensoussan.
\newblock {Mean-Field-Game Model for Botnet Defense in Cyber-Security}.
\newblock {\em Applied Mathematics \& Optimization}, 74(3):669--692, 2016.

\bibitem{KolokoltsovCorruption2017}
V.~N. Kolokoltsov and O.~A. Malafeyev.
\newblock {Mean-Field-Game Model of Corruption}.
\newblock {\em Dynamic Games and Applications}, 7(1):34--47, 2017.

\bibitem{Lacker2023Convergence_Common}
Daniel Lacker and Luc {Le Flem}.
\newblock {Closed-loop convergence for mean field games with common noise}.
\newblock {\em The Annals of Applied Probability}, 33(4):2681--2733, 2023.

\bibitem{laguzet2015individual}
Laetitia Laguzet and Gabriel Turinici.
\newblock {Individual vaccination as Nash Equilibrium in a SIR Model with
  Application to the 2009--2010 Influenza A (H1N1) Epidemic in France}.
\newblock {\em Bulletin of Mathematical Biology}, 77:1955--1984, 2015.

\bibitem{LasryJapanese2007}
Jean-Michel Lasry and Pierre-Louis Lions.
\newblock Mean field games.
\newblock {\em Japanese Journal of Mathematics}, 2(1):229--260, 2007.

\bibitem{lavigne2023decarbonization}
Pierre Lavigne and Peter Tankov.
\newblock Decarbonization of financial markets: a mean-field game approach.
\newblock arXiv:2301.09163, 2023.

\bibitem{Neumann2020}
Berenice~Anne Neumann.
\newblock {Stationary Equilibria of Mean Field Games with Finite State and
  Action Space}.
\newblock {\em Dynamic Games and Applications}, 10:845--871, 2020.

\bibitem{NutzCompetitionRD}
Marcel Nutz and Yuchong Zhang.
\newblock {A Mean Field Competition}.
\newblock {\em Mathematics of Operations Research}, 44(4):1245--1263, 2019.

\bibitem{Teschl2012ODE}
Gerald Teschl.
\newblock {\em Ordinary Differential Equations and Dynamical Systems}, volume
  140 of {\em Graduate Studies in Mathematics}.
\newblock AMS, 2012.

\bibitem{WalterODE1998}
Wolfgang Walter.
\newblock {\em Ordinary Differential Equations}, volume 182 of {\em Graduate
  Texts in Mathematics}.
\newblock Springer-Verlag, 1998.

\end{thebibliography}
	\bibliographystyle{plain}
	
	\appendix
	
	\section{Construction of the Measure Changes} \label{appendix:measure_change}
	
	In this section we detail the construction of the probability measures $\mathbb{P}^\alpha$ in the model formulations in Section~\ref{sec:model} and Section~\ref{sec:approximation}, via suitable changes of measure.
	In particular, we demonstrate that the relevant processes have the desired intensities under those measures.
	
	\subsection{Model Setup with $n$ Shocks}
	\label{sec:measure_change_bounded}
	
	We use the setup defined in Section~\ref{sec:model}.
	For each admissible strategy $\alpha\in\mathcal{A}$ we define the probability measure $\mathbb{P}^\alpha$ via
	\begin{align}
		\label{eq:DefPAlpha}
		\begin{split}
			\frac{\de \mathbb{P}^\alpha}{\de \mathbb{P}} &= \prod_{i, j \in \mathbb{S}: i \neq j} \prod_{k =0}^n \left(  \exp \left\{ \int_0^T (1-Q^{ij}(s,k,M_s,\alpha_s)) \mathbb{I}_{\{Z_s = k\}} \de s \right\} \right.  \\
			&\qquad \left. \cdot \prod_{s \in (0,T]: \Delta N_s^{ikj} \neq 0} \left(Q^{ij}(s,k,M_s,\alpha_s) \mathbb{I}_{\{Z_s = k\}} + \mathbb{I}_{\{Z_s \neq k\}} \right)  \right) \\
			&\quad \cdot \prod_{l =1}^n \left( \exp \left\{ \int_0^T (1-\lambda^l(s,M_s)) \mathbb{I}_{\{Z_s = l-1\}}  \de s \right\} \right. \\
			&\qquad \left. \cdot \prod_{s \in (0,T]: \Delta N_s^l \neq 0} \left( \lambda^l(s,M_s) \mathbb{I}_{\{Z_s = l-1\}}  + \mathbb{I}_{\{Z_s \neq l-1\}} \right) \right).
		\end{split}
	\end{align}
	
	\begin{remark}
		As a technical sidenote, observe that we change the intensities of the counting processes $N^{k}$ and $N^{ikj}$ progressively, i.e.\ only at the times that are actually relevant for the dynamics of $Z$ or $X$.
		In the setup of Section~\ref{sec:model}, it would also be possible to modify all jump intensities at once (this would, of course, not affect the mean field equilibrium); however, in the model setup of Section~\ref{sec:approximation} with an unlimited number of shocks this is not feasible.
		For consistency, we therefore use the definition that works in both models. \remEnde
	\end{remark}
	
	The following proposition summarizes the relevant properties of $\mathbb{P}^\alpha$, including the fact that the processes $N^k$ and $N^{ikj}$ have the desired intensities.
	
	\begin{proposition}
		\label{Thm:ChangeOfMeasure}
		For every admissible strategy $\alpha\in\mathcal{A}$, $\mathbb{P}^\alpha$ is a well-defined probability measure on $(\Omega,\mathfrak{A})$ with $\mathbb{P}^\alpha = \mathbb{P}$ on $\sigma(X_0)$. 
		Moreover, for each $k\in\{1,\ldots,n\}$ the process $N^k$ is a counting process with $(\mathfrak{F},\mathbb{P}^\alpha)$-intensity $\lambda^k=(\lambda^k_t)_{t \in [0,T]}$,
		\[
		\lambda^k_t \defined \lambda^k(t,M_t) \mathbb{I}_{\{Z_t=k-1\}} + \mathbb{I}_{\{Z_t\neq k-1\}},\quad t\in[0,T]
		\]
		and for all $i,j\in\mathbb{S}$ and $k\in\{0,1,\ldots,n\}$ with $i \neq j$ the process $N^{ikj}$ is a counting process with $(\mathfrak{F},\mathbb{P}^\alpha)$-intensity $\lambda^{ikj}=(\lambda^{ikj}_t)_{t \in [0,T]}$,
		\[
		\lambda^{ikj}_t \defined Q^{ij}(t,k,M_t,\alpha_t) \mathbb{I}_{\{Z_t=k\}} + \mathbb{I}_{\{Z_t\neq k\}},\quad t\in[0,T].
		\]
	\end{proposition}
	
	In particular $\mathbb{P}^\alpha(\Delta N_t^{ikj} \neq 0) = 0$ and $\mathbb{P}^\alpha(\Delta N^k_t \neq 0) = 0$ for all $t \in [0,T]$, so we have $\Delta X_t = 0$ and $\Delta Z_t = 0$ $\mathbb{P}^\alpha$-a.s.\ for all $t\in[0,T]$.
	
	\begin{proof}
		Since $N^{ikj}$ and $N^k$ are standard Poisson processes under $\mathbb{P}$, the compensated processes $\tilde{N}^{ikj}$ and $\tilde{N}^l$,
		$$\tilde{N}^{ikj}_t \defined N^{ikj}_t - t
		\quad \text{and} \quad 
		\tilde{N}_t^l = N_t^l -t,\quad t\in[0,T]$$
		are $(\mathfrak{F},\mathbb{P})$-martingales for all $i, j \in \mathbb{S}$,  $k \in \{0,1,\ldots, n\}$  and $l \in \{1, \ldots, n\}$ with $i \neq j$.
		We define $\theta^\alpha = (\theta^\alpha_t)_{t \in [0,T]}$ by
		\begin{align*}
			\theta_t^\alpha &\defined \sum_{i, j \in \mathbb{S}: i \neq j} \sum_{k=0}^n \int_0^t \left( Q^{ij}(s,k, M_s, \alpha_s) -1 \right) \mathbb{I}_{\{Z_s=k\}} \de \tilde{N}^{ikj}_s \\
			&\quad + \sum_{l=1}^n \int_0^t \left( \lambda^l(s,M_s) -1\right) \mathbb{I}_{\{Z_s=l-1\}} \de \tilde{N}_s^l, \quad t \in [0,T]. 
		\end{align*}
		The Dol\'{e}ans-Dade exponential $\mathcal{E}[\theta^\alpha]$ is a local $(\mathfrak{F}, \mathbb{P})$-martingale with
		\begin{align*}
			\mathcal{E}[\theta^\alpha]_t &= \prod_{i, j \in \mathbb{S}: i \neq j} \prod_{k =0}^n \left( \exp \left\{ \int_0^t (1- Q^{ij}(s,k,M_s,\alpha_s)) \mathbb{I}_{\{Z_s=k\}} \de s \right\} \right. \\
			&\qquad \left. \cdot \prod_{s \in (0,t]: \Delta N_s^{ikj} \neq 0} \left(Q^{ij}(s,k, M_s, \alpha_s) \mathbb{I}_{\{Z_s=k\}} + \mathbb{I}_{\{Z_s \neq k\}}  \right) \right) \\
			&\quad \cdot \prod_{l =1}^n \left( \exp \left\{ \int_0^t (1- \lambda^l(s,M_s)) \mathbb{I}_{\{Z_s = l-1\}} \de s \right\} \right. \\
			&\qquad \left. \cdot \prod_{s \in (0,t] \Delta N_s^l \neq 0} (\lambda^l(s,M_s) \mathbb{I}_{\{Z_s = l-1\}} + \mathbb{I}_{\{Z_s \neq l-1\}}) \right), \quad t \in [0,T].
		\end{align*}
		Writing 
		\begin{align*}
			Y_1 &\defined \sum_{i,j \in \mathbb{S}: i \neq j, k \in \{0,1, \ldots, n\}} N_T^{ikj} \sim_{\mathbb{P}} \text{Poi} (S(S-1)(n+1) T) \\
			Y_2 &\defined  \sum_{l=1}^n N_T^l \sim_{\mathbb{P}} \text{Poi} (nT) 
		\end{align*} 
		we obtain
		\[
		\sup_{t \in [0,T]} | \mathcal{E}[\theta^\alpha]_t | \le e^{S^2 (n+1) T} \cdot (\max\{Q_\text{max}, 1\})^{Y_1}  \cdot e^{nT} \cdot (\max \{\lambda_\text{max},1\})^{Y_2}.
		\] 
		Since $\sup_{t \in [0,T]} | \mathcal{E}[\theta^\alpha]_t |$ is integrable, it follows that $\mathcal{E}[\theta^\alpha]$ is an $(\mathfrak{F},\mathbb{P})$-martingale.
		Noting that $\mathcal{E}[\theta^\alpha]_0 = 1$ by construction, we conclude that $\mathbb{P}^\alpha$ is a well-defined probability measure on $\mathfrak{A}$ with density process 
		\[
		\left. \frac{\de \mathbb{P}^\alpha}{\de \mathbb{P}} \right|_{\mathfrak{F}_t} = \mathcal{E}[\theta^\alpha]_t \quad\text{for } t \in [0,T].
		\]
		Since $\mathcal{E}[\theta^\alpha]_0=1$, we obtain $\frac{\de \mathbb{P}^\alpha}{\de \mathbb{P}}|_{\mathcal{F}_0}=1$ and therefore $\mathbb{P}^\alpha = \mathbb{P}$ on $\sigma(X_0)$.
		
		Let $i,j \in \mathbb{S}$ and $k \in \{0,1, \ldots, n\}$ such that $i \neq j$. Since $\mathbb{P}^\alpha \ll \mathbb{P}$ the process $N^{ikj}$ is a $\mathbb{P}^\alpha$-counting process. 
		Therefore, it suffices to show that $\overline{N}^{ikj} = (\overline{N}^{ikj}_t)_{t \in [0,T]}$,
		\[
		\overline{N}^{ikj}_t \defined N_t^{ikj} - \int_0^t \left(Q^{ij}(s,k, M_s, \alpha_s) \mathbb{I}_{\{Z_s=k\}} + \mathbb{I}_{\{Z_s \neq k\}} \right) \de s, \quad t \in [0,T]
		\]
		is a local $(\mathfrak{F}, \mathbb{P}^\alpha)$-martingale.
		By the Bayes rule it suffices to show that $\Theta^\alpha \cdot \overline{N}^{ikj}$ is a local $(\mathfrak{F}, \mathbb{P})$-martingale, where $\Theta^\alpha \defined \mathcal{E}[\theta^\alpha]$.
		To establish this, note that by the product rule
		\[
		\de \left(\Theta_t^\alpha \cdot \overline{N}_t^{ikj}\right)
		= \Theta_{t-}^\alpha \de \overline{N}^{ikj}_t + \overline{N}^{ikj}_{t-} \de \Theta_t^\alpha + \de \left[\Theta^\alpha, \overline{N}^{ikj} \right]_t.
		\]
		Using that $N^{ikj}$ and any other process out of
		\[
		(N^{i'k'j'})_{i',j' \in \mathbb{S}, k' \in \mathbb{W}: i' \neq j' \wedge (i',k',j') \neq (i,k,j)} \quad \text{and} \quad (N^l)_{l \in \{1, \ldots,n\}}
		\]
		do almost surely not jump simultaneously, it follows that
		\[
		\de \left[\Theta^\alpha, \overline{N}^{ikj} \right]_t = \Theta_{t-}^\alpha (Q^{ij}(t,k,M_t, \alpha_t)-1) \mathbb{I}_{\{Z_t=k\}} \de N^{ikj}_t.
		\]
		Thus we have
		\begin{align*}
			\de \left(\Theta_t^\alpha \cdot \overline{N}_t^{ikj}\right)
			&= \Theta_{t-}^\alpha \de N^{ikj}_t - \Theta_{t-}^\alpha \left(Q^{ij}(t,k,M_t, \alpha_t) \mathbb{I}_{\{Z_t=k\}} + \mathbb{I}_{\{Z_t\neq k\}} \right) \de t + \overline{N}^{ikj}_{t-} \de \Theta_t^\alpha \\
			&\quad  + \Theta_{t-}^\alpha Q^{ij}(t,k,M_t, \alpha_t) \mathbb{I}_{\{Z_t= k\}} \de N^{ikj}_t - \Theta_{t-}^\alpha \mathbb{I}_{\{Z_t= k\}} \de N^{ikj}_t \\
			&= \Theta_{t-}^\alpha \mathbb{I}_{\{Z_t \neq k\}} \de \overline{N}_t^{ikj} + \overline{N}^{ikj}_{t-} \de \Theta_t^\alpha + \Theta_{t-}^\alpha Q^{ij}(t,k,M_t, \alpha_t) \mathbb{I}_{\{Z_t= k\}} d \overline{N}^{ikj}_t
		\end{align*}
		and $\Theta^\alpha \cdot \overline{N}^{ikj}$ is indeed a local $(\mathfrak{F}, \mathbb{P})$-martingale.
		
		An analogous argument applies to $N^l$ for $l\in\{0,1,\ldots,n\}$.
		It is clear that $N^l$ is a $\mathbb{P}^\alpha$-counting process, and it remains to show that $\overline{N}^l=\{\overline{N}_t^l \}$,
		\[
		\overline{N}_t^l \defined N_t^l - \int_0^t \left(\lambda^l(s,M_s) \mathbb{I}_{\{Z_s = l-1\}} + \mathbb{I}_{\{Z_s \neq l-1\}} \right) \de s,\quad t\in[0,T]
		\]
		is a local $(\mathfrak{F},\mathbb{P}^\alpha)$-martingale.
		As above, since there are no simultaneous jumps we obtain 
		\begin{align*}
			d \left( \Theta_t^\alpha \cdot \overline{N}_t^l \right)
			&= \Theta^\alpha_{t-} \de \overline{N}_t^l + \overline{N}_{t-}^l \de \Theta_t^\alpha + \de \left[ \Theta^\alpha, \overline{N}^l \right]_t \\
			&= \Theta^\alpha_{t-} \de N_t^l - \Theta_{t-}^\alpha \left(\lambda^l(t,M_t) \mathbb{I}_{\{Z_t = l-1\}} + \mathbb{I}_{\{Z_t \neq l-1\}} \right) \de t + \overline{N}_{t-}^l \de \Theta_t^\alpha \\
			&\quad + \Theta_{t-}^\alpha \lambda^l(t,M_t) \mathbb{I}_{\{Z_t = l-1\}} \de N_t^l - \Theta_{t-}^\alpha  \mathbb{I}_{\{Z_t = l-1\}} \de N_t^l \\
			&= \Theta_{t-}^\alpha \mathbb{I}_{\{Z_t \neq l-1\}} \de \overline{N}_t^l +\overline{N}_{t-}^l \de \Theta_t^\alpha + \Theta_{t-}^\alpha \lambda^l(t,M_t) \mathbb{I}_{\{Z_t = l-1\}} \de \overline{N}_t^l
		\end{align*}
		and thus $\Theta^\alpha \cdot \overline{N}^l$ is a local $(\mathfrak{F}, \mathbb{P})$-martingale.
	\end{proof}
	
	\begin{remark}
		\label{rem:change_of_measure_two_step}
		The probability measure $\mathbb{P}^\alpha$ can be seen as the result of two subsequent changes of measure via 
		\[
		\frac{\de \mathbb{P}^\alpha}{\de \mathbb{P}} = \prescript{}{Z}{\Theta^\alpha} \cdot \prescript{}{X}{\Theta^\alpha}
		\]
		where 
		\begin{align*}
			\prescript{}{Z}{\Theta^\alpha} &= \prod_{l =1}^n \left( \exp \left\{ \int_0^T (1- \lambda^l(s,M_s)) \mathbb{I}_{\{Z_s = l-1\}} \de s \right\} \right. \\
			&\qquad \left. \cdot \prod_{s \in (0,t] \Delta N_s^l \neq 0} (\lambda^l(s,M_s) \mathbb{I}_{\{Z_s = l-1\}} + \mathbb{I}_{\{Z_s \neq l-1\}}) \right) \\
			\prescript{}{X}{\Theta^\alpha} &= \prod_{i, j \in \mathbb{S}: i \neq j} \prod_{k =0}^n \left(  \exp \left\{ \int_0^T (1-Q^{ij}(s,k,M_s,\alpha_s)) \mathbb{I}_{\{Z_s = k\}} \de s \right\} \right.  \\
			&\qquad \left. \cdot \prod_{s \in (0,T]: \Delta N_s^{ikj} \neq 0} \left(Q^{ij}(s,k,M_s,\alpha_s) \mathbb{I}_{\{Z_s = k\}} + \mathbb{I}_{\{Z_s \neq k\}} \right)  \right).
		\end{align*}
		In particular, given a fixed process $\{M_t\}_{t \in [0,T]}$,  both $ \prescript{}{Z}{\Theta^\alpha}$ and $\prescript{}{X}{\Theta^\alpha}$ are martingales. Hence, we can view the change of measure as a two-step procedure: Using $\prescript{}{Z}{\Theta^\alpha}$ we change the intensities of $N^k$ to be as desired. Thereafter, using $\prescript{}{X}{\Theta^\alpha}$, we change the intensities of $ N^{ikj}$ to be as desired. This immediately implies, that  the choice of the strategy $\alpha$ does not influence the evolution of the common noise, i.e. the distribution of $\prescript{}{Z}{N}^k$ is the same for any probability measure $\mathbb{P}^\alpha$, $\alpha \in \mathcal{A}$. \remEnde
	\end{remark}

	\subsection{Model Setup with an Unbounded Number of Shocks}
	\label{sec:measure_change_unbounded}
	
	In the following we use the notation of Section~\ref{sec:approximation}.
	Similarly as in \eqref{eq:DefPAlpha}, for each admissible strategy $\alpha\in\mathcal{A}$ we define $\mathbb{P}^{\alpha,\infty}$ by  
	\begin{align*}
		\frac{\de \mathbb{P}^{\alpha, \infty}}{\de \mathbb{P}} &= \prod_{i, j \in \mathbb{S}: i \neq j} \prod_{k =0}^{Z_T} \left(  \exp \left\{ \int_0^T (1-Q^{ij}(s,k,M_s,\alpha_s)) \mathbb{I}_{\{Z_s = k\}} \de s \right\} \right.  \\
		&\qquad \left. \cdot \prod_{s \in (0,T]: \Delta N_s^{ikj} \neq 0} \left(Q^{ij}(s,k,M_s,\alpha_s) \mathbb{I}_{\{Z_s = k\}} + \mathbb{I}_{\{Z_s \neq k\}} \right)  \right) \\
		&\quad \cdot \prod_{l =1}^{Z_T} \left( \exp \left\{ \int_0^T (1-\lambda^l(s,M_s)) \mathbb{I}_{\{Z_s = l-1\}}  \de s \right\} \right. \\
		&\qquad \left. \cdot \prod_{s \in (0,T]: \Delta N_s^l \neq 0} \left( \lambda^l(s,M_s) \mathbb{I}_{\{Z_s = l-1\}}  + \mathbb{I}_{\{Z_s \neq l-1\}} \right) \right).
	\end{align*}
	Comparing this with the definition \eqref{eq:DefPAlpha} of $\mathbb{P}^{\alpha,n}$ (denoted by $\mathbb{P}^\alpha$ in Appendix~\ref{sec:measure_change_bounded}), we see that the only difference is that the product extends to $Z_T$ instead of $n$. 
	Hence, it follows immediately that $\frac{\de\mathbb{P}^{\alpha, \infty}}{\de\mathbb{P}} = \frac{\de \mathbb{P}^{\alpha,n}}{\de \mathbb{P}}$ on $\{Z_T \le n\}$.
	
	\begin{remark}
		The change of measure is well-defined since the products in the definition of $\frac{\de\mathbb{P}^{\alpha, \infty}}{\de\mathbb{P}}$ have only finitely many factors a.s.
		Indeed, $Z$ is again a unit-intensity Possion process and hence $Z_T$ is a.s.\ finite.\remEnde
	\end{remark}
	
	We next provide the analog of Proposition~\ref{Thm:ChangeOfMeasure}, showing that $N^k$ and $N^{ikj}$ have the desired intensities under $\mathbb{P}^{\alpha,\infty}$.
	
	\begin{proposition}
		For every admissible strategy $\alpha\in\mathcal{A}$, $\mathbb{P}^{\alpha,\infty}$ is a well-defined probability measure on $(\Omega,\mathfrak{A})$ with $\mathbb{P}^{\alpha,\infty} = \mathbb{P}$ on $\sigma(X_0)$.
		Moreover, for all $k\in\{1,\ldots,n\}$ the process $N^k$ is a counting process with $(\mathfrak{F}, \mathbb{P}^{\alpha, \infty})$-intensity $\lambda^k=(\lambda^k_t)_{t \in [0,T]}$,
		\[
		\lambda^k_t \defined \lambda^k(t,M_t) \mathbb{I}_{\{Z_t=k-1\}} + \mathbb{I}_{\{Z_t\neq k-1\}},\quad t\in[0,T]
		\]
		and for all $i,j\in\mathbb{S}$ and all $k\in\{0,1,\ldots,n\}$ with $i \neq j$ the process $N^{ikj}$ is a counting process with $(\mathfrak{F},\mathbb{P}^{\alpha,\infty})$-intensity $\lambda^{ikj}=(\lambda^{ikj}_t)_{t \in [0,T]}$,
		\[
		\lambda^{ikj}_t \defined Q^{ij}(t,k,M_t,\alpha_t) \mathbb{I}_{\{Z_t=k\}} + \mathbb{I}_{\{Z_t\neq k\}}, \quad t \in [0,T].
		\]
		Furthermore, we have $\frac{\de\mathbb{P}^{\alpha,\infty}}{\de\mathbb{P}} = \frac{\de\mathbb{P}^{\alpha,n}}{\de\mathbb{P}}$ on $\{Z_T \le n\}$.
	\end{proposition}
	
	\begin{proof}
		Most steps are analogous to the proof of Proposition~\ref{Thm:ChangeOfMeasure}, hence we only detail the relevant adjustments.
		
		The definition of $\theta^\alpha$ is slightly adjusted to read
		\begin{align*}
			\theta_t^\alpha &\defined \sum_{i, j \in \mathbb{S}: i \neq j} \sum_{k=0}^{Z_t} \int_0^t \left( Q^{ij}(s,k, M_s, \alpha_s) -1 \right) \mathbb{I}_{\{Z_s=k\}} \de \overline{N}^{ikj}_s \\
			&\quad + \sum_{l=1}^{Z_t} \int_0^t \left( \lambda^l(s,M_s) -1\right) \mathbb{I}_{\{Z_s=l-1\}} \de \overline{N}_s^l, \quad t \in [0,T]. 
		\end{align*} 
		While the process defined in the proof of Proposition~\ref{Thm:ChangeOfMeasure} is clearly a local martingale, here we require an additional argument to deal with the random sums
		\[
		\sum_{k=0}^{Z_t} \int_0^t (Q^{ij}(s,k,M_s, \alpha_s)-1) \mathbb{I}_{\{Z_s = k\}} \de \overline{N}_s^{ikj}
		\] 
		and 
		\[
		\sum_{l=1}^{Z_t} \int_0^t (\lambda^l(s,M_s)-1)) \mathbb{I}_{\{Z_s = l-1\}} \de \overline{N}_s^l.
		\]
		We demonstrate this for the first sum since the proofs are completely analogous.
		For $n\in\mathbb{N}$ define the process $L^n = (L_t^n)_{t \in [0,T]}$,
		\[
		L_t^n \defined \sum_{k=0}^n \int_0^t (Q^{ij}(s,k,M_s,\alpha_s)-1) \mathbb{I}_{\{Z_s=k\}} \de \overline{N}_s^{ikj}, \quad t \in [0,T].
		\]
		It is clear that each $L^n$ is a local martingale, $n\in\mathbb{N}$; since we have
		\[
		\mathbb{E} \left[ \sup_{t \in [0,T]} |L_t^n| \right] \le (Q_\text{max} +1) \mathbb{E} [Y] + T
		\] 
		where $Y = \sum_{k=0}^\infty \mathbb{I}_{\{Z_s=k\}} N_T^{ikj} \sim_{\mathbb{P}} \text{Poi}(T)$ it follows that $L^n$ is in fact a martingale.
		Since the processes $(L^n)_{n\in\mathbb{N}}$ converge almost surely uniformly to the process $L= \{L_t\}$,
		\[
		L_t \defined \sum_{k=0}^{Z_t} \int_0^t (Q^{ij}(s,k, M_s, \alpha_s)-1) \mathbb{I}_{\{Z_s = k\}} \de \overline{N}_s^{ikj}
		\]
		dominated convergence implies that $L$ is a martingale.	
		
		The construction of $\mathcal{E}[\theta^\alpha]$ as a Dol\'{e}ans-Dade exponential is literally the same as in the proof of Proposition~\ref{Thm:ChangeOfMeasure}.
		It is clear that $\mathcal{E}[\theta^\alpha]$ is a local martingale; to demonstrate that it is a true martingale, we again show that $\sup_{t \in [0,T]} |\mathcal{E}[\theta^\alpha]_t|$ is integrable.
		This requires an additional argument.
		Setting
		\begin{align*}
			\tilde{Y} &\defined \sum_{k=0}^{Z_T} \tilde{Y}_k \quad \text{with} \quad \tilde{Y}_k =\sum_{i,j \in \mathbb{S}:i \neq j}  N_T^{ikj} \overset{\text{iid}}{\sim_{\mathbb{P}} } \text{Poi}(S(S-1)T)
		\end{align*} 
		we obtain
		\[
		\sup_{t \in [0,T]} |\mathcal{E}[\theta^\alpha]_t| \le e^{S^2T} \cdot \left( \max \{Q_\text{max},1\}\right)^{\tilde{Y}} \cdot e^T \cdot \left( \max \{\lambda_\text{max},1\}\right)^{Z_T}.
		\] 
		Using that $(\tilde{Y}_k)_{k \in \mathbb{N}_0}$ an $Z_T$ are independent, we obtain for any $a_1, a_2 \in \mathbb{R}$ that
		\begin{align*}
			\mathbb{E} \left[ e^{a_1 \tilde{Y}} e^{a_2 Z_T} \right] 
			&=\mathbb{E} \left[ \sum_{n=0}^\infty \mathbb{I}_{\{Z_T = n\}} \exp \left(a_1 \sum_{k=0}^{Z_T} \tilde{Y}_k \right) e^{a_2 Z_T} \right] \\
			&= \sum_{n=0}^\infty \mathbb{E} \left[  \mathbb{I}_{\{Z_T = n\}} \exp \left(  a_1 \sum_{k=0}^n \tilde{Y}_k \right) e^{a_2 n} \right] \\
			&= \sum_{n=0}^\infty \mathbb{P} (Z_T = n)  \mathbb{E} \left[ \exp \left(  a_1 \sum_{k=0}^n \tilde{Y}_k \right) \right] e^{a_2 n} \\
			&= \sum_{n=0}^\infty \mathbb{P} (Z_T = n) \exp \left( S(S-1)T (e^{a_1}-1) \right)^{n+1} e^{a_2 n} \\
			&= \exp \left(S(S-1)T (e^{a_1}-1) \right) \mathbb{E} \left[ \exp \left( Z_T \left( S(S-1)T(e^{a_1}-1) + a_2 \right) \right) \right] \\
			&= \exp \left(S(S-1)T (e^{a_1}-1) \right) \exp \left( T \exp ( S(S-1)T(e^{a_1}-1)+a_2) -1 \right) < \infty
		\end{align*} hence the upper bound is indeed integrable.

		The rest of the proof, including the verification of the relevant jump intensities, is completely analogous to the proof of Proposition~\ref{Thm:ChangeOfMeasure}.
	\end{proof}
	
	\section{Proof of the Existence and Uniqueness Theorem}
	\label{appendix:existence}
	
	In this section we show that under suitable assumptions there exists a unique solution to the forward-backward system \eqref{eq:FB1}-\eqref{eq:FB5} on a small time horizon. The proof is based on Banach's fixed point theorem.
	However, it is not possible to imitate the standard proofs that are available for mean field games without common shocks for two reasons: First, it is not clear whether a solution of \eqref{eq:FB1} subject to \eqref{eq:FB3} exists, since the function $v^j(s,\overrightarrow{(u,s)})$ shows up in the right-hand side and it is a priori not clear whether it is measurable in $s$. Second, proving that the function that maps $\mu$ to the solution of \eqref{eq:FB1} and \eqref{eq:FB3} as well as the function that maps $v$ to the solution of \eqref{eq:FB2}, \eqref{eq:FB4} and \eqref{eq:FB5} is Lipschitz and computing its Lipschitz constant is also not straightforward due to the hierarchical structure of the solutions.
	
	\begin{assumption}
		\label{assumption:Banach}
		\begin{itemize}
			\item[(i)] There is a constant $L_{\hat{\psi}}>0$ such that the reduced-form running reward function $\hat{\psi}$ satisfies
			\[
			||\hat{\psi}(t,u, m_1,v_1) - \hat{\psi}(t,u,m_2,v_2)|| \le L_{\hat{\psi}} (||m_1-m_2||+||v_1-v_2||) 
			\] for all $(t,u) \in \text{TS}$, all $m_1,m_2 \in \mathcal{P}(\mathbb{S})$ and all $v_1,v_2 \in \mathbb{R}^S$. 
			\item[(ii)] There is a constant $L_{\hat{Q}}>0$ such that the reduced-form intensity matrix function $\hat{Q}$ satisfies
			\[
			||\hat{Q}(t,u,m_1,v_1) - \hat{Q}(t,u,m_2,v_2)|| \le L_{\hat{Q}} \left( ||m_1 - m_2|| + ||v_1-v_2|| \right)
			\] for all $(t,u) \in \text{TS}$, $m_1, m_2 \in \mathcal{P}(\mathbb{S})$ and all $v_1,v_2 \in \mathbb{R}^S$.
			\item[(iii)] There is a constant $L_\Psi>0$ such that the terminal reward function satisfies
			\[
			||\Psi(k,m_1) - \Psi(k,m_2)|| \le L_\Psi ||m_1-m_2||
			\]
			for all $k \in \{0, 1, \ldots, n\}$ and all $m_1, m_2 \in \mathcal{P}(\mathbb{S})$.
			\item[(iv)] There is a constant $L_\lambda>0$ such that
			\[
			|\lambda^k(t,m_1) - \lambda^k(t,m_2)| \le L_\lambda ||m_1-m_2||
			\] for all $t \in [0,T]$, $k \in \{1, \ldots, n \}$ and $m_1, m_2 \in \mathcal{P}(\mathbb{S})$.
			\item[(v)] The relocation function satisfies
			\[
			J^i(t,m_1)=J^i(t,m_2)
			\]
			for all $t \in [0,T]$, $i \in \mathbb{S}$ and $m_1, m_2 \in \mathcal{P}(\mathbb{S})$.
		\end{itemize}
	\end{assumption}
	
	To formulate the precise statement let us introduce the following notation for the maximal values of the functions specifying rewards and dynamics:
	\begin{equation}
		\label{eq:appendix_constants}
		\begin{aligned}
			\Psi_{\text{max}} &\defined \sup_{\substack{k \in \{0,1, \ldots, n\} \\ m \in \mathcal{P}(\mathbb{S})}} ||\Psi(k,m)|| &
			\psi_\text{max} &\defined \sup_{\substack{t \in [0,T], k \in \{0,1, \ldots, n\} \\ m \in \mathcal{P}(\mathbb{S}), a \in \mathbb{A}}} ||\psi(t,k,m,a)|| \\
			Q_\text{max} &\defined \sup_{\substack{t \in [0,T], k \in \{0,1, \ldots, n\} \\ m \in \mathcal{P}(\mathbb{S}), a \in \mathbb{A}}} ||Q(t,k,m,a)|| &
			\lambda_\text{max} &\defined \sup_{\substack{t \in [0,T], k \in \{1, \ldots, n\} \\ m \in \mathcal{P}(\mathbb{S})}} \lambda^k(t,m) \\
			J_\text{max} &\defined \sup_{t \in [0,T], m \in \mathcal{P}(\mathbb{S}), i \in \mathbb{S}} |\{j : J^i(t,m)=j\}|.    
		\end{aligned}
	\end{equation}
	Moreover, we introduce the following constants:
	\begin{align*}
		v_\text{max} &\defined \left( \Psi_\text{max} + \psi_\text{max} T\right) \exp  \left( \left( Q_\text{max} + \lambda_\text{max} \right) T \right) \sum_{i=0}^{n} \left( \exp \left( \left( Q_\text{max} + \lambda_\text{max} \right) T \right) \lambda_\text{max} T \right)^i \\
		K_1 &\defined L_{\hat{\psi}} + L_{\hat{Q}} v_\text{max} + 2 v_\text{max} L_\lambda \\
		K_2 &\defined L_{\hat{\psi}} + v_\text{max} L_{\hat{Q}} + Q_\text{max} + \lambda_\text{max}.
	\end{align*}
	
	\begin{theorem}
		\label{Theorem:ExistenceUniqueness}
		Let Assumption~\ref{assumption:Banach} hold and assume that $T$ satisfies 
		\[
		L_{\hat{Q}} T \left( \sum_{i=0}^n (\exp (Q_\text{max} + L_{\hat{Q}}T))^{i+1} (J_{\text{max}})^i \right) (L_\psi + K_1 T) \exp(K_2 T) \left( \sum_{i=0}^n (\exp(K_2T) \lambda_\text{max} T)^i \right) <1.
		\]
		Then there exists a unique solution of \eqref{eq:FB1}-\eqref{eq:FB5}.
	\end{theorem}
	
	Since all norms on $\mathbb{R}^S$ are equivalent the concrete choice of the norm $||\cdot||$ in Assumption~\ref{assumption:Banach} is immaterial. To simplify our calculations in the following we will use the maximum norm and a compatible matrix norm. Moreover, let us define the space \[
	\mathcal{M} \defined  \{ \mu: \text{TS} \rightarrow \mathcal{P}(\mathbb{S}) | \; \mu \text{ is } Q_\text{max}\text{-Lipschitz continuous in } t \}
	\] and equip it with the uniform norm $||\mu|| \defined \sup_{(t,u) \in \text{TS}} ||\mu(t,u)||$. Since $B(\text{TS}, \mathcal{P}(\mathbb{S}))$ is a complete metric space and $\mathcal{M}$ is a closed subset, $\mathcal{M}$ is a complete metric space.
	We set 
	\[
	\mathcal{V} \defined \{v:\text{TS} \rightarrow \mathbb{R}^S | \; v \text{ is absolutely continuous in } t \text{ and satisfies } ||v|| \le v_\text{max}\}.
	\]
	
	For ease of reference we state the classical forward Gronwall estimate (e.g., \cite[Lemma 2.7]{Teschl2012ODE}) as well as a corresponding backward estimate.
	\begin{lemma}
		\label{lemma:appendix_gronwall}
		[Gronwall estimates] Assume that $f: [0,T] \rightarrow \mathbb{R}$ is continuous, $C \in \mathbb{R}$ and $\beta \ge 0$. Then:
		\begin{itemize}
			\item[(i)] (forward estimate) If $f(t) \le C + \beta \int_0^t f(s) \de s$, $t \in [0,T]$, then $f(t) \le C \exp (\beta t)$ for all $t \in [0,T]$.
			\item[(ii)] (backward estimate) If $f(t) \le C + \beta \int_t^T f(s) \de s$, $t \in [0,T]$, then $f(t) \le C \exp (\beta(T-t))$ for all $t \in [0,T]$.
		\end{itemize}
	\end{lemma}

	\begin{lemma}
		\label{lemma:appendix_existence_v}
		Let Assumption~\ref{assumption:Banach} hold and let $\mu \in \mathcal{M}$ be given. Then there is a unique solution of \eqref{eq:FB1} and \eqref{eq:FB3}. Moreover, $||v(t,u)|| \le v_\text{max}$, i.e., $v \in \mathcal{V}$.
	\end{lemma}
	
	\begin{proof}
		We construct the unique solution recursively for all $u$ such that $Z(u)=k$ starting with $k=n$:
		For $u \in \mathbb{R}^n$ such that $Z(u)=n$ the equation \eqref{eq:FB1} reads 
		\[
		\frac{\partial}{\partial t} v(s,u) = - \hat{\psi} (s,u, \mu(s,u),v(s,u)) - \hat{Q}(s,u,\mu(s,u), v(s,u)) v(s,u)
		\] and the terminal condition reads $v(T,u)=\Psi(n,\mu(T,u))$. By Assumption~\ref{assumption:Banach} it is now immediate that the right-hand side is locally bounded, measurable in $s$ and locally Lipschitz-continuous in $v(s,u)$. Thus, by \cite[Theorem 10.XX]{WalterODE1998} a unique solution $v(\cdot, u):[u_n,T] \rightarrow \mathbb{R}^S$ exists.
		Let us remark at this point that if we choose $\mu(s,u)$ for $s <u_n$ appropriately such that $\mu$ remains continuous then by the same argument a unique solution $\tilde{v}(\cdot, u):[0,T] \rightarrow \mathbb{R}^S$ exists, and it coincides with the solution derived above on $[u_n,T]$.
		
		Let us now assume that $k<n$ and that solutions for all $u$ with $Z(u)>k$ have been constructed. Let $u$ be arbitrary such that $Z(u)=k$. In this setting equation \eqref{eq:FB1} reads 
		\begin{align*}
			\frac{\partial}{\partial t} v(s,u) &= - \hat{\psi} (s,u, \mu(s,u),v(s,u)) - \hat{Q}(s,u,\mu(s,u), v(s,u)) v(s,u) \\
			&\quad - \lambda^{k+1} (s,\mu(s,u)) \left( \left( v^{J^i(s,\mu(s-,u))}(s, \overrightarrow{(u,s)}) \right)_{i \in \mathbb{S}} - v(s,u) \right)
		\end{align*} and the terminal condition reads $v(T,u)=\Psi(k,\mu(T,u))$. Whereas it is still immediate that the right-hand-side is locally bounded and locally Lipschitz continuous in $v(T,u)$, it is necessary in order to achieve measurability in $s$ that the function $$s \mapsto \left( v^{J^i(s,\mu(s-,u))}(s, \overrightarrow{(u,s)}) \right)_{i \in \mathbb{S}}$$ is measurable. By Assumption~\ref{assumption:Banach} it suffices to prove that $s \mapsto v^i(s,\overrightarrow{(u,s)})$ is continuous in $s$:
		For each $s> u_k$ we choose $\mu(t,\overrightarrow{(u,s)}) = \mu(t,u)$ for all $u_k<t<s$. Then as discussed above we obtain a unique solution $\tilde{v}(\cdot, u):[u_k,T] \rightarrow \mathbb{R}^S$ that coincides on $[s,T]$ with the solution derived for $\overrightarrow{(u,s)}$ in the preceding step. However, we can view these solutions $\tilde{v}(\cdot,\overrightarrow{(u,s)})$ as solutions of an ordinary differential equation with parameter $s$. Noting that the right-hand-side of the ordinary differential equation is continuous in both variables $s$ and $v$, we obtain that the solutions $\tilde{v}(t,\overrightarrow{(u,s)})$ depends continuously on the parameter $s$ \cite[Theorem 1.1.6]{FilippovDiscontiODE1988}. Thus, in particular the function $s \mapsto \tilde{v}(s,\overrightarrow{(u,s)})$ is continuous in $s$, which is the desired claim. 
		
		Therefore, again all conditions of Theorem 10.XX in \cite{WalterODE1998} are satisfied and we obtain the existence of  a unique solution $v(\cdot, u):[u_k,T] \rightarrow \mathbb{R}^S$. As before let us remark that if we choose $\mu(s,u)$ for $s <u_k$ appropriately such that $\mu$ remains continuous, then, by the same argument, a unique solution $\tilde{v}(\cdot, u):[0,T] \rightarrow \mathbb{R}^S$ exists and coincides on $[u_k,T]$ with the solution derived before.

		Let us finally prove that $||v(t,u)||$ is bounded by $v_\text{max}$ for all $(t,u)\in \text{TS}$. Let $(t,u) \in \text{TS}$ be arbitrary, then
		\begin{align*}
			||v(t,u)|| &\le ||\Psi(Z(u), \mu(T,u))|| + \int_t^T || \hat{\psi}(s,u, \mu(s,u), v(s,u)) || \de s \\
			&\quad + \int_t^T ||\hat{Q}(s,u,\mu(s,u),v(s,u))|| ||v(s,u)|| \de s \\
			& + \sum_{l=1}^n \mathbb{I}_{\{Z(u)=l-1\}} \int_t^T |\lambda ^l(s, \mu(s,u))| \left( ||(v^{J^i(s,\mu(s-,u))}(s, \overrightarrow{(u,s)})))_{i \in \mathbb{S}}|| + || v(s,u)|| \right) \de s  \\
			&\le \Psi_\text{max} + \psi_\text{max} T + \int_t^T (Q_\text{max} + \lambda_\text{max}) ||v(s,u)|| \de s \\
			&\quad + \sum_{l=1}^n \mathbb{I}_{\{Z(u)=l-1\}} \int_t^T \lambda_\text{max} ||v(s,\overrightarrow{(u,s)}) || \de s.
		\end{align*}
		We will now prove by backward induction on $Z(u)$ that 
		\[
		||v(t,u)|| \le \left( \Psi_\text{max} + \psi_\text{max} T\right) \exp  \left( \left( Q_\text{max} + \lambda_\text{max} \right) T \right) \sum_{i=0}^{n-Z(u)} \left( \exp \left( \left( Q_\text{max} + \lambda_\text{max} \right) T \right) \lambda_\text{max} T \right)^i 
		\] for all $ t \ge \max_{k \in \{1, \ldots, n\}} u_k$:
		If $u$ is such that $Z(u)=n$, then we have
		\[
		||v(t,u)|| \le \left( \Psi_\text{max} + \psi_\text{max} T\right) + \int_t^T \left(Q_\text{max} + \lambda_\text{max} \right) ||v(s,u)|| \de s		
		\] and the claim immediately follows from the backward Gronwall estimate (Lemma~\ref{lemma:appendix_gronwall}). if $u$ is such that $Z(u)=k<n$, then the induction hypothesis yields
		\begin{align*}
			&||v(t,u)|| \\
			&\le \left( \Psi_\text{max} + \psi_\text{max} T\right) + \int_t^T \left( Q_\text{max} + \lambda_\text{max} \right) ||v(s,u) || \de s \\
			&\quad + \lambda_\text{max} T \left( \Psi_\text{max} + \psi_\text{max} T \right) \exp \left( \left( Q_\text{max} + \lambda_\text{max} \right) T \right) \sum_{i=0}^{n-(k+1)} \left( \exp \left( \left( Q_\text{max} + \lambda_\text{max} \right) T \right) \lambda_\text{max} T \right)^i \\
			&= \left( \Psi_\text{max} + \psi_\text{max} T \right)  \sum_{i=0}^{n-k} \left( \exp \left( \left( Q_\text{max} + \lambda_\text{max} \right) T \right) \lambda_\text{max} T \right)^i \\
			&\quad + \int_t^T \left( Q_\text{max} + \lambda_\text{max} \right) ||v(s,u)|| \de s, 
		\end{align*} from which, using the backward Gronwall estimate (Lemma~\ref{lemma:appendix_gronwall}), we directly deduce the desired claim.
	\end{proof}
	
	\begin{lemma}
		\label{lemma:appendix_existence_mu}
		Let Assumption \ref{assumption:Banach} hold and let $v \in \mathcal{V}$ be given. Then there exists a unique solution of \eqref{eq:FB2}, \eqref{eq:FB4} and \eqref{eq:FB5}, which furthermore satisfies $\mu \in \mathcal{M}$.
	\end{lemma}
	
	\begin{proof}
		We construct the unique solution inductively for all $u$ such that $Z(u)=k$ starting with $k=0$:
		For $k=0$ we have $u=u_0$ and the ordinary differential equation \eqref{eq:FB2} and \eqref{eq:FB4} is a classical ODE with a right-hand side that is locally bounded, measurable in $s$ and locally Lipschitz-continuous in $v(s,u)$. Thus, again by \cite[Theorem 10.XX]{WalterODE1998} a unique solution exists. Noting that for $k>0$ the ordinary differential equation \eqref{eq:FB2} and \eqref{eq:FB5} now again satisfies these properties we immediately obtain the existence and uniqueness of $\mu$. 
		
		To prove that $\mathcal{P}(\mathbb{S})$ is flow invariant for $\mu$, i.e.\ $\mu(s,u) \in \mathcal{P}(\mathbb{S})$ for all $(s,u) \in \text{TS}$, by Theorem 5.3.1 of \cite{CarjaFlowInvCaratheodory} it suffices to demonstrate that
		\[
		\hat{Q}(s,u, m, v(s,u))^Tm \in T_{\mathcal{P}(\mathbb{S})}(m)
		\] 
		for all $m \in \mathcal{P}(\mathbb{S})$ and all $s \in [0,T]$ and $u$ such that $(s,u) \in \text{TS}$, where $T_{\mathcal{P}(\mathbb{S})}(m)$ denotes the Bouligand tangent cone.
		To show this, note first that for points $m$ lying in the interior we have $T_{\mathcal{P}(\mathbb{S})}(m)=\mathbb{R}^S$, so nothing has to be checked. 
		For $m \in \partial  T_{\mathcal{P}(\mathbb{S})}(m)$ we have $y \in T_{\mathcal{P}(\mathbb{S})}(m)$  if and only if $\sum_{i \in \mathbb{S}} y_i = 0$ and $y_i \ge 0$ for all $i$ such that $m_i=0$. Moreover, for such boundary points the vector $ \hat{Q}(s,u, m, v(s,u))^T m$ has non-negative entries at each $j \in \mathbb{S}$ such that $m_j=0$ because the only non-positive column entry of $\hat{Q}(s,u, m, v(s,u))$ has weight zero. Noting that since $\hat{Q}(s,u, m, v(s,u))$ is a conservative generator
		$$\sum_{j \in \mathbb{S}} \sum_{i \in \mathbb{S}} m_i \hat{Q}^{ij}(s,u, m, v(s,u)) = \sum_{i \in \mathbb{S}} \underbrace{\sum_{j \in \mathbb{S}} m_i \hat{Q}^{ij}(s,u, m, v(s,u))}_{=0} = 0,$$
		it follows that $\hat{Q}(s,u, m, v(s,u))^T m \in T_{\mathcal{P}(\mathbb{S})}(m)$ for all $m \in \mathcal{P}(\mathbb{S})$ and all $(s,u) \in \text{TS}$.
		
		Finally, using that $\mu(s,u) \in \mathcal{P}(\mathbb{S})$ it is immediate that $\mu$ is $Q_\text{max}$-Lipschitz since
		\[
		\left| \left| \frac{\partial}{\partial t} \mu(s,u) \right| \right| \le ||\mu(s,u)|| \left| \left| \hat{Q}(s,u, \mu(s,u),v(s,u)) \right| \right| \le Q_\text{max}.\qedhere
		\]
	\end{proof}
	
	Let us now define $\overleftarrow{F}: \mathcal{M} \rightarrow \mathcal{V}$ as the function that maps $\mu \in \mathcal{M}$ to the unique solution of \eqref{eq:FB1} and \eqref{eq:FB3} given $\mu$. Analogously, we define $\overrightarrow{F}: \mathcal{V} \rightarrow \mathcal{M}$ as the function that maps $v \in \mathcal{V}$ to the unique solution of \eqref{eq:FB2}, \eqref{eq:FB4} and \eqref{eq:FB5} given $v$. By Lemma~\ref{lemma:appendix_existence_v} and Lemma~\ref{lemma:appendix_existence_mu} both mappings are well-defined. Now we prove that these maps are Lipschitz continuous:
	
	\begin{lemma}
		\label{Lemma:LipschitzLeftarrow}
		Let Assumption \ref{assumption:Banach} hold.
		The function $\overleftarrow{F}$ is Lipschitz continuous with Lipschitz constant
		\[
		\left( L_{\Psi} + K_1 T \right) \exp (K_2 T) \cdot \sum_{i=0}^{n} \left( \exp (K_2 T) \lambda_\text{max} T \right)^i.
		\]
	\end{lemma}
	
	\begin{proof}
		Let $\mu, \overline{\mu} \in \mathcal{M}$ and let $v = \overleftarrow{F}(\mu)$ and $\overline{v} = \overleftarrow{F}(\overline{\mu})$. Then we have for fixed $(t,u) \in \text{TS}$:
		\begin{align*}
			&||v(t,u) - \overline{v}(t,u)|| \\
			&\le ||v(T,u) - \overline{v}(T,u)|| \\
			& \quad + \int_t^T ||\hat{\psi}(s,u,\mu(s,u), v(s,u)) - \hat{\psi}(s,u, \overline{\mu}(s,u), \overline{v}(s,u)) || \de s \\
			&\quad + \int_t^T ||\left(\hat{Q}(s,u,\mu(s,u), v(s,u)) v(s,u)) - \hat{Q}(s,u, \overline{\mu}(s,u)), \overline{v}(s,u)) \right) \overline{v}(s,u) || \de s \\
			&\quad + \sum_{l=1}^n \mathbb{I}_{\{Z(u)=l-1\}} \int_t^T \left| \left| \lambda^l(s, \mu(s,u)) \left( (v^{J^i(s, \mu(s-,u))}(s, \overrightarrow{(u,s)}))_{i \in \mathbb{S}} - v(s,u) \right) \right. \right.  \\
			&\qquad \left.\left. - \lambda^l(s, \overline{\mu}(s,u)) \left( (\overline{v}^{J^i(s, \overline{\mu}(s-,u))}(s, \overrightarrow{(u,s)}))_{i \in \mathbb{S}} - \overline{v}(s,u) \right)  \right| \right| \de s \\
			& \le ||\Psi(Z(u), \mu(T,u)) - \Psi(Z(u), \overline{\mu}(T,u)) || \\
			&\quad + \int_t^T ||\hat{\psi}(s,u,\mu(u,s), v(u,s)) - \hat{\psi}(s,u, \overline{\mu}(s,u), \overline{v}(s,u)) || \de s \\
			&\quad + \int_t^T ||\hat{Q}(s,u, \mu(s,u), v(s,u)) - \hat{Q}(s,u, \overline{\mu}(s,u), \overline{v}(s,u)) || \cdot ||v(s,u)|| \de s \\
			&\quad + \int_t^T ||\hat{Q}(s,u, \overline{\mu}(s,u), \overline{v}(s,u))|| \cdot ||v(s,u)- \overline{v}(s,u)|| \de s \\
			&\quad + \sum_{l=1}^n \mathbb{I}_{\{Z(u)=l-1\}} \left( \int_t^T |\lambda^l(s,\mu(s,u)) - \lambda^l(s, \overline{\mu}(s,u))| || (v^{J^i(s, \mu(s-,u))}(s, \overrightarrow{(u,s)}))_{i \in \mathbb{S}} - v(s,u) || \de s \right. \\
			&\qquad  +\int_t^T |\lambda^l(s, \overline{\mu}(s,u))|  || (v^{J^i(s, \mu(s-,u))}(s, \overrightarrow{(u,s)}))_{i \in \mathbb{S}} - (\overline{v}^{J^i(s, \overline{\mu}(s-,u))}(s, \overrightarrow{(u,s)}))_{i \in \mathbb{S}} ||   \de s \\
			&\qquad + \left. \int_t^T  |\lambda^l(s, \overline{\mu}(s,u))| ||v(s,u) - \overline{v}(s,u)|| \de s \right) \\
			&\le L_\Psi ||\mu(T,u)- \overline{\mu}(T,u)|| \\
			&\quad + \int_t^T L_{\hat{\psi}} \left( ||\mu(s,u)- \overline{\mu}(s,u)|| + ||v(s,u)- \overline{v}(s,u)|| \right) \de s \\
			&\quad + \int_t^T v_\text{max} L_{\hat{Q}} \left( ||\mu(s,u)- \overline{\mu}(s,u)|| + ||v(s,u)-\overline{v}(s,u)|| \right) \de s \\
			&\quad + \int_t^T Q_\text{max} ||v(s,u)-\overline{v}(s,u)|| \de s \\
			& + \sum_{l=1}^n \mathbb{I}_{\{Z(u)=l-1\}} \left( \int_t^T 2 v_\text{max} L_\lambda ||\mu(s,u)-\overline{\mu}(s,u)|| \de s \right. \\
			&\qquad \left. + \int_t^T \lambda_\text{max} ||v(s,\overrightarrow{(u,s)}) - \overline{v}(s, \overrightarrow{(u,s)})|| \de s + \int_t^T \lambda_\text{max} ||v(s,u)- \overline{v}(s,u)|| \de s \right) \\
			&\le \left( L_\Psi + TL_{\hat{\psi}} + Tv_\text{max}L_{\hat{Q}} + 2v_\text{max}L_\lambda T \right) ||\mu-\overline{\mu}|| \\
			&\quad + \int_t^T \left( L_{\hat{\psi}} + v_\text{max}L_{\hat{Q}} + Q_\text{max} +  \lambda_\text{max} \right) ||v(s,u)-\overline{v}(s,u) || \de s \\
			&\quad + \sum_{l=1}^n \mathbb{I}_{\{Z(u)=l-1\}} \int_t^T  \lambda_\text{max} ||v(s, \overrightarrow{(u,s)}) - \overline{v}(s,\overrightarrow{(u,s)}) || \de s \\
			&= \left( L_{\Psi} + K_1 T \right) || \mu - \overline{\mu}|| + \int_t^T K_2 ||v(s,u)-\overline{v}(s,u)|| \de s \\
			&\quad + \sum_{l=1}^n \mathbb{I}_{\{Z(u)=l-1\}} \int_t^T  \lambda_\text{max} ||v(s, \overrightarrow{(u,s)}) - \overline{v}(s,\overrightarrow{(u,s)}) || \de s.
		\end{align*}
		We now prove by backward induction on $Z(u)$ that 
		\[
		||v(t,u)- \overline{v}(t,u)|| \le \left( L_\Psi + K_1 T \right) \exp (K_2T) \left( \sum_{i=0}^{n-Z(u)} \left( \exp(K_2T) \lambda_\text{max} T \right)^i \right) ||\mu-\overline{\mu}||
		\] for all $t \ge \max_{k \in \{1,\ldots, n\}} u_k$:
		For $u$ such that $Z(u)=n$ we have
		\[
		||v(t,u)-\overline{v}(t,u)|| \le \left( L_{\Psi} + K_1 T \right) || \mu - \overline{\mu}|| + \int_t^T K_2 ||v(s,u)-\overline{v}(s,u)|| \de s
		\] for all $t \ge \max_{k \in \{0,1,\ldots, n\}} u_k$. Thus, by the backward Gronwall estimate (Lemma~\ref{lemma:appendix_gronwall}) the desired claim follows. For $Z(u)=k<n$ we obtain using the induction hypothesis
		\begin{align*}
			&||v(t,u)- \overline{v})(t,u)|| \\
			&\le \left(L_\Psi + K_1 T\right) ||\mu- \overline{\mu}|| + \int_t^T K_2 ||v(s,u)- \overline{v}(s, u)|| \de s \\
			&\quad + \lambda_\text{max} T \left(  L_\Psi + K_1 T \right) \exp(K_2T) \sum_{i=0}^{n-(k+1)} \left( \exp(K_2T) \lambda_\text{max} T \right)^i || \mu- \overline{\mu}|| \\
			&= \left( L_\Psi + K_1 T \right) \left( 1 + \lambda_\text{max} T \exp(K_2 T) \sum_{i=0}^{n-(k+1)} \left( \exp (K_2 T) \lambda_\text{max} T \right)^i \right) ||\mu - \overline{\mu}|| \\
			& \quad  + \int_t^T K_2 ||v(s,u)- \overline{v}(s, u)|| \de s \\
			&= \left( L_\Psi + K_1 T \right) \left(\sum_{i=0}^{n-k} \left( \exp (K_2 T) \lambda_\text{max} T \right)^i \right)  ||\mu - \overline{\mu}|| \\
			& \quad  + \int_t^T K_2 ||v(s,u)- \overline{v}(s, u)|| \de s
		\end{align*} and, again, the desired claim follows by the backward Gronwall estimate (Lemma~\ref{lemma:appendix_gronwall}).
	\end{proof}
	
	\begin{lemma}
		\label{Lemma:LipschitzRightarrow}
		Let Assumption \ref{assumption:Banach} hold. 
		The function $\overrightarrow{F}$ is Lipschitz continuous with Lipschitz contant 
		\[
		L_{\hat{Q}} T \sum_{i=0}^{n} \left( \exp \left( \left( Q_\text{max} + L_{\hat{Q}} \right) T \right) \right)^{i+1} (J_\text{max})^i.
		\]
	\end{lemma}
	
	\begin{proof}
		Let $v, \overline{v} \in \mathcal{V}$ and let $\mu = \overrightarrow{F}(v)$ and $\overline{\mu} = \overrightarrow{F}(\overline{v})$. Using the convention $u^0 \defined 0$ we obtain for fixed $(t,u) \in \text{TS}$:
		\begin{align*}
			&||\mu(t,u)-\overline{\mu}(t,u)|| \\
			&\le \mathbb{I}_{\{u^1 \neq -1\}}||\mu(u^{Z(u)},u)-\overline{\mu}(u^{Z(u)},u)|| \\
			&\quad + \int_{u^{Z(u)}}^t ||\mu(s,u) \hat{Q}(s,u,\mu(s,u), v(s,u)) - \overline{\mu}(s,u) \hat{Q}(s,u, \overline{\mu}(s,u), \overline{v}(s,u))|| \de s \\
			&\le \mathbb{I}_{\{u^1 \neq -1\}} ||\mu(u^{Z(u)},u)-\overline{\mu}(u^{Z(u)},u)|| \\ 
			&\quad + \int_{u^{Z(u)}}^t ||\mu(s,u) - \overline{\mu}(s,u) || \cdot ||\hat{Q}(s,u,\mu(s,u),v(s,u))|| \de s \\
			&\quad + \int_{u^{Z(u)}}^t ||\overline{\mu}(s,u)|| || \hat{Q}(s,u, \mu(s,u), v(s,u))- \hat{Q}(s,u, \overline{\mu}(s,u), \overline{v}(s,u))|| \de s \\
			& \le \mathbb{I}_{\{u^1 \neq -1\}} ||\mu(u^{Z(u)},u)-\overline{\mu}(u^{Z(u)},u)|| + \int_{u^{Z(u)}}^t Q_\text{max} ||\mu(s,u) - \overline{\mu}(s,u) || \de s \\
			&\quad + \int_{u^{Z(u)}}^t L_{\hat{Q}} \left( ||\mu(s,u)-\overline{\mu}(s,u)|| + ||v(s,u)-\overline{v}(s,u)|| \right) \de s \\
			&\le \mathbb{I}_{\{u^1 \neq -1\}} ||\mu(u^{Z(u)},u)-\overline{\mu}(u^{Z(u)},u)||  + L_{\hat{Q}} T ||v-\overline{v}||\\
			&\quad + \int_{u^{Z(u)}}^t (Q_\text{max} + L_{\hat{Q}}) ||\mu(s,u)-\overline{\mu}(s,u)|| \de s.
		\end{align*}
		We now show by induction on $Z(u)$ that 
		\[
		||\mu(t,u)-\overline{\mu}(t,u)|| \le L_{\hat{Q}} T \sum_{i=0}^{Z(u)} \left( \exp \left( \left( Q_\text{max} + L_ {\hat{Q}} \right) T \right) \right)^{i+1} (J_\text{max})^i ||v-\overline{v}||
		\] 
		for all $t \ge \max_{k \in \{1,\ldots, n\}} u_k$:	
		If $Z(u)=0$, then
		\[
		||\mu(t,u)- \overline{\mu}(t,u)|| \le L_{\hat{Q}}T ||v-\overline{v}|| + \int_0^t \left( Q_\text{max} + L_{\hat{Q}} \right) ||\mu(s,u)- \overline{\mu}(s,u)|| \de s,
		\] 
		from which a simple application of the forward Gronwall estimate (Lemma~\ref{lemma:appendix_gronwall}) yields the desired bound.
		If $Z(u)=k>0$, then using the induction hypothesis
		\begin{align*}
			&||\mu(t,u)- \overline{\mu}(t,u)|| \\
			& \le ||\mu(u^{Z(u)},u)- \overline{\mu}(u^{Z(u)},u)|| + L_{\hat{Q}}T ||v-\overline{v}|| + \int_{u^k}^t (Q_\text{max} +L_{\hat{Q}}) ||\mu(s,u)-\overline{\mu}(s,u)|| \de s \\
			&\le J_\text{max} ||\mu(u^{Z(u)}, \overleftarrow{u}) - \overline{\mu}(u^{Z(u)}, \overleftarrow{u})|| + L_{\hat{Q}}T ||v-\overline{v}|| + \int_{u^k}^t (Q_\text{max} +L_{\hat{Q}}) ||\mu(s,u)-\overline{\mu}(s,u)|| \de s \\
			&\le J_\text{max} L_{\hat{Q}} T \left(\sum_{i=0}^{k-1} \left( \exp \left( \left( Q_\text{max} + L_{\hat{Q}} \right) T \right) \right)^{i+1} (J_\text{max})^i \right) ||v-\overline{v}|| \\
			&\quad + L_{\hat{Q}} T ||v-\overline{v}|| + \int_{u^k}^t (Q_\text{max} +L_{\hat{Q}}) ||\mu(s,u)-\overline{\mu}(s,u)|| \de s \\ 
			&= L_{\hat{Q}} T \left(\sum_{i=0}^{k} \left( \exp \left( \left( Q_\text{max} + L_{\hat{Q}} \right) T \right) \right)^i (J_\text{max})^i \right) ||v-\overline{v}|| + \int_{u^k}^t (Q_\text{max} +L_{\hat{Q}}) ||\mu(s,u)-\overline{\mu}(s,u)|| \de s
		\end{align*} and the desired claim immediately follows using the forward Gronwall estimate (Lemma~\ref{lemma:appendix_gronwall}).
	\end{proof}
	
	\begin{proof}[Proof of Theorem~\ref{Theorem:ExistenceUniqueness}]
		By Lemma~\ref{Lemma:LipschitzLeftarrow} and Lemma~\ref{Lemma:LipschitzRightarrow} the map $F: \mathcal{M} \rightarrow \mathcal{M}$, $\mu \mapsto \overrightarrow{F} \left( \overleftarrow{F} (\mu) \right)$ satisfies for $\mu, \overline{\mu} \in \mathcal{M}$
		\begin{align*}
			&||F(\mu) - F(\overline{\mu})|| \\
			&\le L_{\hat{Q}} T \left( \sum_{i=0}^n (\exp (Q_\text{max} + L_{\hat{Q}}T))^{i+1} (J_{\text{max}})^i \right) \\
			&\quad \cdot (L_\psi + K_1 T) \exp(K_2 T) \left( \sum_{i=0}^n (\exp(K_2T) \lambda_\text{max} T)^i \right) ||\mu-\overline{\mu} ||.
		\end{align*}
		Thus $F$ is contractive provided $T$ is chosen as in the claim. Since $\mathcal{M}$ is a complete metric space, Banach's fixed point theorem yields the existence of a unique fixed point. This in turn is equivalent to \eqref{eq:FB1}-\eqref{eq:FB5} having a unique solution.
	\end{proof}
	
	\noindent\textbf{Funding:} The authors thank the German Research Foundation DFG for their support within the Research Training Group 2126 Algorithmic Optimization (ALOP). \\
	
	\noindent\textbf{Competing Interests:} The authors have no relevant financial or non-financial interests to disclose.
	
\end{document}